\documentclass[12pt,reqno]{amsart}
\usepackage{euscript}
\usepackage{amssymb}
\usepackage{latexsym}
\usepackage{amsmath}
\usepackage{enumitem}

   \def\sH{{\mathfrak H}}   
   \def\sK{{\mathfrak K}}   \def\sL{{\mathfrak L}}
\def\sM{{\mathfrak M}}   \def\sN{{\mathfrak N}}

      \def\dC{{\mathbb C}}

   \def\dN{{\mathbb N}}   
      \def\dR{{\mathbb R}}

\def\cA{{\mathcal A}}   \def\cB{{\mathcal B}}   
\def\cD{{\mathcal D}}      \def\cF{{\mathcal F}}
   \def\cH{{\mathcal H}}   
   \def\cK{{\mathcal K}}   
\def\cM{{\mathcal M}}   \def\cN{{\mathcal N}}   
      
\def\cS{{\mathcal S}}

\def\bA{{\mathbf A}}   \def\bB{{\mathbf B}}

\def\dim{{\rm dim\,}}
\def\ran{{\rm ran\,}}

\def\dom{{\rm dom\,}}

\def\codim{{\rm codim\,}}
\def\lspan{{\rm span\,}}

\def\uphar{{\upharpoonright\,}}

\def\f{\varphi}

\newtheorem{theorem}{Theorem}[section]

\newtheorem{proposition}[theorem]{Proposition}
\newtheorem{corollary}[theorem]{Corollary}

\newtheorem{open}[theorem]{Open problem}
\newtheorem{example}[theorem]{Example}
\numberwithin{equation}{section}

\theoremstyle{remark}
\newtheorem{remark}{Remark}[section]

\def\RE{{\rm Re\,}}
\def\IM{{\rm Im\,}}

\def\wt{\widetilde}
\def\wh{\widehat}

\newcommand\I{{\rm i}}
\newcommand\e{{\rm e}}
\def\sa{\mathfrak a}
\def\sb{\mathfrak b}
\newcommand\bz{\overline{z}}
\newcommand\D{{\rm D}_x}
\usepackage{color}

\begin{document}
\title
[Everything is possible for {\small $\dom T\cap\dom T^*$}]
{Everything is possible for \\ the domain intersection {\large \boldmath $\dom T\cap\dom T^*$}}
\author[Yury Arlinski\u{\i}]{Yury Arlinski\u{\i}}
\address{Volodymyr Dahl East Ukrainian National University, 
Deparment of Mathematics, 
pr.\ Central 59-A,
Severo\-donetsk, 93400, Ukraine}
\email{yury.arlinskii@gmail.com}

\author[Christiane Tretter]{Christiane Tretter}
\address{Mathematisches Institut, Universit\"{a}t Bern,
Sidlerstr.\ 5, 3012, Bern, Switzerland} \email{tretter@math.unibe.ch} \dedicatory{}

\thanks{C.\  Tretter gratefully acknowledges funding of the Swiss National Science Foundation (SNF), grant no.\  $169104$.
\hspace*{3.5mm}
Y.\ Arlinski\u{\i} thanks the Mathematical Institute at the University of Bern for funding and kind hospitality.}

\vskip 1truecm
\thispagestyle{empty}

\date{\today}

\dedicatory{To the memory of our colleague and academic teacher Prof.\ Dr.\ Reinhard Mennicken}

\begin{abstract}
This paper shows that for the domain intersection $\dom T\cap\dom T^*$ of a closed linear operator and its Hilbert space adjoint
everything is possible for very common classes of operators with non-empty resolvent set. Apart from the most striking case of a maximal
sectorial operator with $\dom T\cap\dom T^*\!=\!\{0\}$, we construct classes of operators for which $\dim(\dom T\cap\dom T^*)\!=\! n \!\in\! \dN_0$;
$\dim(\dom T\cap\dom T^*)\!= \infty$ and at the same time $\codim(\dom T\cap\dom T^*)\!=\!\infty$; and $\codim(\dom T\cap\dom T^*)\!= n \!\in\! \dN_0$;
the latter includes~the case that $\dom T\cap\dom T^*$ is dense but no core of $T$ and $T^*$ and the case $\dom T\!=\!\dom T^*$ for non-normal $T$. We also show that all these possibilities may occur for operators $T$ with non-empty resolvent set such that either $W(T)\!=\!\dC$, $T$ is maximal accretive but not sectorial, or $T$ is even maximal sectorial. Moreover, in  all but one subcase $T$ can be chosen with compact 
resolvent.
\end{abstract}

\maketitle



\section{Introduction}

Questions on the intersection of the domains of two unbounded operators have seen several surprising answers, including
von Neumann's theorem \cite{MR1581206} and Kato's square root problem \cite{MR0138005}, \cite{MR0151868}.
While the former establishes the existence of unbounded selfadjoint operators $A$ and $B$
in a complex infinite dimensional separable Hilbert space such that the extreme case $\dom A \cap \dom B \!=\! \{0\}$ holds,
the latter concerns the other extreme case whether $\dom A^{1/2} \!=\! \dom {A^{1/2}}^*$ for maximal accretive or even maximal sectorial operators~$A$.  Kato \cite{MR0138005,MR0151868} had shown that if $A$ is maximal accretive, then $\dom A^{\alpha} \!=\! \dom {A^{\alpha}}^*$ for $\alpha \!\in\! [0,\frac 12)$, whereas inequality may hold for $\alpha \!\in\! (\frac 12,1]$, and $\dom A^{1/2} \cap \dom {A^{1/2}}^*$ is a core of both $A^{1/2}$ and ${A^{1/2}}^*\!$, without any assump\-tions on the relation between $\dom A$ and $\dom A^*$.
Shortly after, Lions \cite{MR0152878} established two sufficient conditions for equality and showed that Kato's maximal accretive example for inequality for $\alpha \in (\frac 12,1]$ also satisfies $\dom A^{1/2} \ne \dom {A^{1/2}}^*$. Later independent counter-examples of McIntosh \cite{MR0290169} and of Gomilko \cite{MR939523} showed that even for maximal sectorial (also called regularly accretive) operators the domain inequality $\dom A^{1/2}
\ne\dom {A^{1/2}}^*$ may hold.

In this paper we consider the more general problem of the ``size'' of the domain intersection $\dom T\cap\dom T^*$ for arbitrary closed linear operators $T$ with non-empty resolvent set. In fact, the domain intersection $\dom T \cap\dom T^*$ is an interesting indicator
to assess the deviation of an unbounded non-selfadjoint operator $T$ from symmetric or normal operators $S$ since for the latter, by definition, $\dom S \subset \dom S^*$  and hence $\dom S \cap\dom S^* = \dom S$ is always dense. So a natural question to ask is: \emph{Does
the ``size'' of $\dom T\cap\dom T^*$ depend on properties of $T$ such as being maximal sectorial, maximal accretive or having numerical range $W(T)=\dC$ equal to the entire complex plane}?

Domain intersections $\dom T \cap\dom T^*$ play an important role in the theory of unbounded non-selfadjoint operators. Examples include the density problems for domains of self commu\-tators and anticommutators $T^*T \pm TT^*$ \cite{MR257789}, \cite{MR1974034}, \cite{MR2138701}, \cite{MR3405900}, defining a real and imaginary part of a closed operator in terms of the operator sums $\frac 12 (T\!+\!T^*)$, $\frac 1{2\I} (T\!-\!T^*)$ \cite{MR0367708},
the relation between relative boundedness resp.\
compactness and relative form-boundedness resp.\ form-compactness \cite{MR2501173},
an unbounded version \cite{MR3162253} of Hildebrandt's theorem \cite{MR0200725} on extremal boundary points of the numerical ran\-ge,
and equivalent descriptions of the essential numerical range \cite{MR4083777}. However, so far there~are no results on $\dom T \cap\dom T^*\!$ beyond obvious cases where $\dom T \cap\dom T^*\!=\dom T$ is automatically dense by~definition such as unbounded symmetric $T$ or ``almost'' normal~$T$. Here the latter refers to Putnam \cite{MR338818} (rather than to ``almost normal'' in the sense of \cite{MR3310942})
who used this term in the bounded case to summarize quasinormal, subnormal, hyponormal, seminormal etc.\ operators,
which have been studied in the unbounded case as well
\cite{MR3192032}, 
\cite{MR1018684}, 
\cite{MR1028066}, \cite{MR1307601}, 
\cite{MR257789}. 

In this paper we give a whole series of unexpected answers to the question highlighted above which may be briefly put as `\emph{everything is possible for $\dom T \cap \dom T^*$}'\!,
even for nice classes of operators such as maximal sectorial operators~$T$. Our main results show that the extreme case $\dom T\cap\dom T^*=\{0\}$, the case of arbitrary
finite dimension $\dim(\dom T\cap\dom T^*)\!=\! n \in \dN_0$,
the case of both infinite dimension and infinite codimension $\dim(\dom T\cap\dom T^*)\!= \infty$ and $\codim(\dom T\cap\dom T^*)\!=\!\infty$,
the case of arbitrary finite codimension $\codim(\dom T\cap\dom T^*)\!= n \in \dN$, and the case that $\dom T\cap\dom T^*$ is dense but neither a core of $T$ nor of $T^*$ may prevail.
At the same time we prove that all these possibilities may occur for very different classes of operators with non-empty resolvent set,
such as maximal sectorial operators, maximal accretive operators that are not sectorial, and operators with maximal numerical range $W(T)=\dC$.
Moreover, we show that it is possible to choose $T$ with compact resolvent in all but one subcase.
Our constructions do not only yield particular examples or counter-examples, they rather provide classes of operators for which $\dom T\cap\dom T^*$ exhibits these unexpected phenomena.

There are only a few existing results on the relation between the domains $\dom T$ and $\dom T^*$, and almost all are restricted to the cases that $\dom T=\dom T^*$ and $\dom T\cap\dom T^*$ is a core of $T$. In \cite{MR2828331} and \cite{MR1383821},
for some classes of quasi-selfadjoint extensions $T$ of a non-densely defined symmetric operator,
the domain equality $\dom T=\dom T^*$ was established.  In \cite[Prop.\ 3.5, 3.2]{MR760619}
it was shown that for a closed densely defined operator $T$ with $\dom T^2=\dom T$,
which necessitates that $W(T)=\dC$, the intersection $\dom T\cap\dom T^*$ is a core of $T$.
In \cite{MR2897729} and \cite{MR4083777} examples of operators with codim $(\dom T\cap\dom T^*) =1$ were given which arise as singular perturbations of selfadjoint operators.

In this paper we  completely unfold the much richer picture that may arise.
We provide abstract constructions of classes of densely defined operators $T$
in an infinite dimensional, for most cases separable complex Hilbert space
for every possible combination of the three operator types
\\[1mm]
\begin{tabular}{rl}
(I) &$T$ maximal sectorial, \\
(II) &$T$ maximal accretive but not sectorial, \\
(III) &$W(T)=\dC$, $\rho(T)\ne\emptyset$,
\end{tabular}
\\[1mm]
with the following seven possible phenomena for $\dom T\cap\dom T^*$:
\\[1mm]
\begin{tabular}{rl}
(1) &$\dom T\cap\dom T^*\!=\!\{0\}$,\\
(2) &$\dim(\dom T\cap\dom T^*) =\!n$ for arbitrary $n\!\in\!\dN$,\\
(3) &
    $
    \dim(\dom T\cap\dom T^*)\!=\!\infty, \, 
		\codim(\dom T\cap\dom T^*)\!=\!\infty$,
		\\
(4) &$\codim(\dom T\cap\dom T^*) =\!n$ for arbitrary $n\!\in\!\dN$,\\
(5) &$\dom T\!\nsubseteq\!\dom T^*,\ \dom T\cap\dom T^*\;\mbox{dense, but not a core of } T$,\\
(6) &$
\dom T\!\nsubseteq\!\dom T^*,\ \dom T\cap\dom T^*\;\mbox{core of } T$,\\
\ \ (7) &$\dom T=\dom T^*$,  \, \mbox{but $T$ non-normal},
\end{tabular}
\pagebreak

Our main tools include
von Neumann's theorem \cite[Satz 18]{MR1581206}, stated in Theorem \ref{vnthm} below,  and its equivalent formulations \cite{MR3612999},  the representation of closed sectorial sesquilinear forms and operators associated with them \cite{MR1335452}, special classes of maximal accretive extensions of closed sectorial operators \cite{MR1772627},
as well as non-symmetric singular perturbations of selfadjoint operators \cite{MR3496031}.
Note that Kato's result \cite[Thm.\ 5.1]{MR0151868}
yields a class of maximal sectorial $T$ such that $\dom T \cap \dom T^*$ is a core of $T$ and $T^*$ (case (I.6) above), namely
$T\!=\!A^{1/2}$ with $A$ maximal~accretive.

The paper is organized as follows. Section \ref{sec:prelim} contains some preliminaries, including von~Neu\-mann's theorem.
In Section~\ref{sec:max-sect-dim=n} we present maximal sectorial operators for which $\dom T\cap\dom T^*\!=\!\{0\}$ (case (I.1) above) and
$\dim(\dom T\cap\dom T^*) \!=\!n \!\in\! \dN$ (case (I.2)~above). In Section \ref{sect.infinite} we study maximal sectorial operators $T$
such that $\codim(\dom T\cap\dom T^*)\!=\!\infty$  and $\dom T\cap\dom T^*$ is a proper closed subspace of finite or infinite codimension (case (I.3) above). In Section~\ref{sec:extensions-S_z} we consider families $S_z$, $\RE z\le 0,$ of maximal accretive (and possibly maximal sectorial) extensions of a non-negative densely defined symmetric operator $S$ and prove  that
\begin{itemize}
\item[$+$] if $\RE z\!<\!0$ and $\IM z\!\ne\!0,$ then $\dom S_z\cap\dom S^*_z$ is dense and no core of $S_z$ and~$S^*_z$,
\item[$+$] if $z=a<0$, then $\dom S_a\cap\dom S^*_a$  is dense in $(S-aI)\dom S =\sN_a^\perp$,
\item[$+$] if $z=\I x,$ $x\in\dR\setminus\{0\},$ then 
$\dom S_{\I x}\cap\dom S^*_{\I x}=\sN_{\I x}\oplus (S+\I x I)\dom S^2$,
\end{itemize}
where $\sN_z =\ker (S^*-zI)$ is the defect subspace of $S$ corresponding to  $z \!\in\! \dC$.
Varying the properties of $\sN_a$, $\sN_{\I x}$ and $\dom S^2$ by the choice of $S$, we obtain a whole series of answers
for maximal sectorial and maximal accretive $T$ (cases (I.3), (I.4), (I.5), (II.3), (II.4), (II.5) above).
In Section~\ref{section:singperturb}, by means of singular perturbations of an unbounded selfadjoint operator,
we construct classes of operators for which $\codim(\dom T\cap\dom T^*) =\!n \!\in\! \dN$ and $W(T)\!=\!\dC$ (case (III.4) above).
In Section~\ref{sec:equality}, following an approach in \cite{MR627724}, we derive maximal
accretive and maximal sectorial operators whose domain coincides with that of its adjoint, $\dom T=\dom T^*$ (cases (I.7) and (II.7) above).
In Section~\ref{sec:additive-perturb} we establish stability results which allow us to construct 1) maximal accretive non-sectorial operators
from maximal sectorial operators and  2) operators with full numerical range $W(T)\!=\!\dC$ from sectorial or accretive operators,
but preserving all seven possible cases of of the domain intersection (cases (II.2), (II.3), (II.6) and (III.1)--(III.7), respectively, above).
In Section \ref{sec:more-ex}, by means of a family of fundamental symmetries $J(z)$, $z\!\in\!\dC$, which is continuous in the operator-norm topology, we construct families of closed densely defined operators $\wh T (z)$, $z\in\dC$, possessing the properties $\dom \wh T(z)\cap\dom \wh T^*(z)\!=\!\{0\}$, $W(\wh T(z))\!=\!W(\wh T(z)^*)=\dC$ for all $z\!\in\!\dC$.
Moreover,  we show that all possible phenomena for $\dom T\cap\dom T^*$ occur even for operators of all the above three classes (I), (II), (III) with compact resolvent
except for one case;  if $\dom T\cap \dom T^*$ is an infinite dimensional closed subspace, then $T$ can never have compact resolvent by the closed graph theorem.  Finally, in Section \ref{last} we construct various holomorphic operator families of type (A) and (B), e.g.\ of the form $\Psi(z) =
A^* (I+T(z))A$ with $A$ maximal sectorial and associated semigroup $T(z)$, for which $\dom \Psi(z) \cap \dom \Psi(z)^*$ may either be dense or~$\{0\}$.

{\bf Notations.}
We use the symbols $\dom T$, $\ran T$, $\ker T$ for
the domain, the range, and the kernel of a linear operator $T$;
the closures of $\dom T$, $\ran T$ are denoted by $\overline{\dom T}$,
$\overline{\ran T}$, respectively. The identity operator in a Hilbert space
$\sH$ is denoted by  $I$ and sometimes by $I_\sH$. If $\sL \subset \sH$ is a
closed subspace, the orthogonal projection in $\sH$ onto $\sL$ is denoted by~$P_\sL.$ The notation
$T\uphar \cN$ means the restriction of a linear operator $T$ to a subspace
$\cN\subset\dom T$. The resolvent set of $T$ is denoted by $\rho(T)$.
The space of bounded linear operators acting between Hilbert spaces $\sH$ and $\sK$ is denoted
by $\bB(\sH,\sK)$ and the Banach algebra $\bB(\sH,\sH)$ by
$\bB(\sH)$.  Finally, $\dC$ and $\dR$ denote the set of complex and real numbers, respectively, $\dR_+:=[0,+\infty)$,
and $\dN$ is the set of natural numbers, $\dN_0:=\dN\cup\{0\}$.
In order to avoid confusion with closures, for $\Omega \subset \dC$, we denote by $\Omega^*:=\{\overline z \in \dC: z \in \Omega\}$
the complex conjugate of $\Omega$.

\section{Preliminaries and von Neumann's theorem}
\label{sec:prelim}

In this section we present some
operator theoretic ingredients which we use the in following such as von Neumann's theorem on domain intersections and several
equivalent formulations.

\smallskip

\emph{Preliminaries.} \
The \emph{numerical range} is the tool by means of which we classify the different classes of operators we consider.
For a linear operator $T$ with domain $\dom T$ in a Hilbert space $\sH$ with scalar product $(\cdot,\cdot)$ it is given by
\[
  W(T)=\left\{(Tu,u):u\in\dom T, \,\|u\|=1\right\}.
\]
As is well-known \cite[Thm.~V.3.2]{MR1335452}, the numerical range is a convex set and has the spectral inclusion property
$\sigma_{\rm p}(T)\subset W(T)$ and $\sigma_{\rm app}(T) \subset \overline{W(T)}$ for the point spectrum and approximate point spectrum of $T$.
Moreover, the range $\ran (T-\lambda I)$ is closed for every $\lambda \in \dC\setminus \overline{W(T)}$ and its dimension is locally constant;
this implies that $\sigma(T) \subset \overline{W(T)}$ if each of the (at most two) components of $\dC\setminus \overline{W(T)}$ contains a point
of the resolvent set $\rho(T)$. In this case, the numerical range provides the resolvent \vspace{-1mm} estimate
\begin{equation}
\label{nr-res-est}
  \|(T-\lambda I)^{-1}\| \le \frac 1 {{\rm dist}\,(\lambda, W(T))}, \quad \lambda \in \dC \setminus \overline{W(T)}.
\vspace{-1mm}
\end{equation}

A linear operator $\cA$ in a Hilbert space $\sH$ is called \textit{accretive}
if its numerical range lies in the closed right half-plane
\[
  W(\cA) \subset \overline{\dC_+} := \{ z\in\dC: \RE z \ge 0\},
\]
i.e.\ $\RE(\cA u,u)\ge 0$ for all $u\in\dom \cA$.
An accretive operator $\cA$ is called \textit{maximal accretive}, or
\textit{$m$-accretive} for short, if one of the following equivalent conditions is satisfied:
\begin{enumerate}
\def\labelenumi{\rm (\roman{enumi})}
\item $\cA$ has no proper accretive extensions in $\sH$;
\item $\cA$ is densely defined and $\ran (\cA-\lambda I)=\sH$ for some $\lambda\in \dC$ with $\RE\lambda<0$;
\item $\cA$ is densely defined and closed, and $\cA^*$ is accretive;
\item $-\cA$ generates contractive one-parameter semigroup $T(t)=\exp(-t\cA)$, $t\ge 0$.
\end{enumerate}
Besides, if $\cA$ is maximal accretive, then $\ker \cA=\ker\cA^*$ and hence
\begin{equation}
\label{kernel-accr}
   \ker \cA\subseteq\dom\cA\cap\dom\cA^*.
\end{equation}
A linear operator $\cA$ in a Hilbert space $\sH$ is called \textit{dissipative}, or \textit{$m$-dissipative} for short, if
$-\I \cA$ is accretive or $m$-accretive, respectively; in this case the numerical range of $\cA$ is contained in
the closed upper half-plane, $W(\cA)\subset \{ z\in\dC: \IM z \ge 0\}$.

The resolvent set $\rho(A)$ of an $m$-accretive operator contains the open left half-plane
$\dC_- \!:=\! \left\{ z\!\in\! \dC: \right.$ $\left. \RE z < 0\right\}$ and, by \eqref{nr-res-est},
\[
  \|(\cA-\lambda I)^{-1}\|\le\cfrac{1}{|\RE \lambda|},  \quad \RE \lambda<0.
\]
An accretive operator $\cA$ is called \emph{coercive} if there exists $m>0$ with $\RE (\cA f,f)\ge m\|f\|^2$ for all $f\in\dom \cA$.

A linear operator $\cA$ in a Hilbert space $\sH$ is called \emph{sectorial} with vertex $z=0$ and
semi-angle $\alpha \in [0,\pi/2)$, or \emph{$\alpha$-sectorial} for short, if its numerical range is contained in a
closed sector with semi-angle~$\alpha$,
\begin{equation}
\label{ctrnjhf}
  W(\cA) \subset \overline{\mathcal{S}(\alpha)}:=\left\{z\in\dC:|\arg z|\le \alpha\right\}
\end{equation}
or, equivalently, $|\IM (\cA u,u)|\!\le\! \tan\alpha \,\RE(\cA u,u)$ for all $u\!\in\!\dom \cA$.
Clearly, a sectorial operator is accretive; it is called \textit{maximal sectorial}, or \textit{$m$-$\alpha$-sectorial} for short, if
it is $m$-accretive.

The resolvent set of an $m$-$\alpha$-sectorial operator $\cA$ contains the
set $\dC\setminus \overline{\mathcal{S}(\alpha)}$ and, by \eqref{nr-res-est},
\[
 \|(\cA-\lambda I)^{-1}\| \le \cfrac{1}{{\rm dist}\left(\lambda,\mathcal{S}(\alpha)\right)},
 \quad \lambda\in \dC\setminus \overline{\mathcal{S}(\alpha)}.
\]

We mention that if
$\cA$ is $m$-accretive, then for each $\gamma\!\in\!(0,1)$ the fractional powers $\cA^\gamma$ are defined \cite{MR1335452}, \cite{MR0151868}.
The operators $\cA^{\gamma}$ are $m$-sectorial with semi-angle $\gamma\pi/2$ and, if $\gamma\!\in\! (0,1/2)$, then
$\dom \cA^{\gamma}\!=\! \dom \cA^{*\gamma}$\!.  It was proved in \cite[Thm.\ 5.1]{MR0151868} that, if $\cA $ is $m$-sectorial, then
$\dom \cA^{1/2}\!\cap\!\dom \cA^{*1/2}$ is a core of both $\cA^{1/2}$  and $ \cA^{*1/2}$
and the real part $\RE \cA^{1/2}\!:=\!(\cA^{1/2}\!+\!\cA^{*1/2})/2 $ defined on $\dom \cA^{1/2}\!\cap\!\dom \cA^{*1/2}$ is a selfadjoint operator.
Further,~by~\cite[Cor.~2]{MR0151868},
\begin{equation}
\label{ravhalf}
\dom \cA=\dom \cA^* \implies \dom \cA^{1/2}=\dom \cA^{*1/2}=\dom \cA^{1/2}_R = \cD[\sa],
\end{equation}
where $\sa$ is the closed form associated with the sectorial operator $\cA$ via the first representation theorem \cite[Sect.~VI.2.1]{MR1335452}
and $\cA_R$ is the non-negative selfadjoint operator associated with the real part of ${\mathfrak a}$ given by
${\rm Re}\, {\mathfrak a} := ({\mathfrak a} + {\mathfrak a}^*)/2$.

\smallskip

\emph{Von Neumann's theorem.} \
One of the main ingredients for our constructions is the following well-known theorem of J.~von Neumann \cite[Satz 18]{MR1581206}:

\begin{theorem}[of von Neumann]
\label{vnthm}
\textit{If $A$ is an unbounded selfadjoint operator in a separable infinite dimensional complex Hilbert space, then there exists a unitary operator $U$ such that $U\dom A\cap\dom A=\{0\}$, i.e.\ the domains of the two unitarily equivalent selfadjoint operators $A$ and $A'=U^{-1}AU$ have trivial intersection}.
\end{theorem}

Note that von Neumann's theorem does not hold if the Hilbert space is non-separable, as shown recently in \cite{MR3509134}.

In another recent paper \cite[Thm.~5.1]{MR3612999} 
the following equivalent formulations of von Neumann's theorem were established:

\smallskip

\textit{For an unbounded selfadjoint operator $A$ in an infinite dimensional complex and not necessarily separable Hilbert $\sH$
the following are equivalent:
\begin{enumerate}
\def\labelenumi{\rm (\roman{enumi})}
\item there exists a unitary operator $U$ in $\sH$ such that
$$\dom(U^*AU)\cap\dom A=\{0\};$$
\item there exists an unbounded selfadjoint operator $B$ in $\sH$ such that
$$\dom B\cap\dom A=\{0\};$$
\item there exists a fundamental symmetry $J$ in $\sH$ $(J=J^*=J^{-1})$ such that
$$\dom(JAJ)\cap\dom A=\{0\};$$
\item there exists a subspace $\sM$ in $\sH$ such that
$$\sM\cap\dom A=\sM^\perp\cap\dom A=\{0\};$$
\item there exists a closed densely defined restriction $A_0$ of $A$ such that $\dom (AA_0)=\{0\}$
$($and thus, in particular, $\dom A^2_0=\{0\})$.
 \end{enumerate}
 }

\smallskip

Special examples of selfadjoint operators $A$ and $B$ with $\dom A\cap\dom B=\{0\}$ may be found in 
\cite{MR683379}, \cite{MR2216946}; note that the example given in \cite{MR1885442} contains a mistake, see Remark~\ref{mistake}.

\smallskip

If, in von Neumann's theorem, we set
\begin{alignat*}{2}
   &\,T:=UA, \quad &&{\,\dom T=\dom A,}\\
\intertext{then, since $U$ is bounded,}
   &T^*\!=A U^*\!,\quad  &&  \dom T^*\!= U\dom A.
\end{alignat*}
It follows that there exists a densely defined operator $T$ such that $\dom T\cap\dom T^*=\{0\}$,
but we do not obtain any information on the properties of $T$. We will return to this example in Section~\ref{sec:more-ex}.


\section{Maximal sectorial operators $\cA$ such that $\dim(\dom \cA\cap\dom \cA^*)\in \dN_0$}
\label{sec:max-sect-dim=n}

In this section we present the most striking and extreme phenomenon of $m$-sectorial operators such that $\dom \cA\cap\dom \cA^*=\{0\}$ and, more generally, $\dom \cA\cap\dom \cA^*$ is a sub\-space of arbitrary finite dimension $n\in\dN_0$.

Throughout this section we assume that $\sH$ is a separable infinite dimensional complex Hilbert space.

\begin{theorem}
\label{subsec:max-sect-dim=0}
There exists an $m$-sectorial operator $A$ in $\sH$ with $\dom A\cap\dom A^*\!=\{0\}$; \linebreak
the latter necessitates $\ker A \!=\!\{0\}$, and $A$ may be chosen with or without compact resolvent.
\end{theorem}

\begin{proof}
Let $L$ be an unbounded selfadjoint operator in $\sH$ such that
\begin{equation}\label{zero}
   \ker L=\{0\}.
\end{equation}
According to von Neumann's theorem, Theorem \ref{vnthm}, more precisely, claim (ii) equivalent to it,
there exists a bounded selfadjoint operator $G$ in $\sH$~with
\begin{equation}
\label{random0}
    \ker G=\{0\},\quad \ran G\cap\dom L=\{0\};
\end{equation}
e.g.\  we can choose $G=(B^*B+I)^{-1/2}$ in (ii).
 We consider the sesquilinear form
 \[
   \hspace{2mm} {\mathfrak a}[u,v]=((I+\I G)Lu,Lv),\quad \ u,v\in\dom  {\mathfrak a}:=\dom L.
 \]
The form $ {\mathfrak a}$ is densely defined, closed and sectorial with adjoint form $ {\mathfrak a}^*$ given by
\[
   {\mathfrak a}^*[\phi,\psi]=((I-\I G)L\phi,L\psi),\quad \phi,\psi\in\dom  {\mathfrak a}^*\!=\dom L.
\]
By the first representation theorem \cite[Sect.~VI.2.1]{MR1335452}, 
the associated $m$-sectorial operators $A$ and $A^*$ are given by
\begin{alignat*}{2}
\dom A \, &=\left\{u\in\dom L:(I+\I G)Lu\in\dom L\right\}, \quad & Au \ &=L(I+\I G)Lu,\ \ u\in\dom A, \\
\dom A^*\!&=\left\{\phi\in\dom L:(I-\I G)L\phi\in\dom L\right\}, \quad & A^*\!\phi&=L(I-\I G)L\phi,\ \ \phi\in\dom A^*.
\end{alignat*}
This yields the characterization
\begin{equation}
\label{intersection}
 \dom A\cap\dom A^*=\big\{ u \in \dom L : u \in \dom L^2, \, G L u \in \dom L \big\}.
\end{equation}
Due to properties \eqref{random0}, \eqref{zero}
of the operators $L$ and $G$, this implies $\dom A\cap\dom A^*=\{0\}$.

The other claims follow because $\ker A = \ker L = \{0\}$, the operator $L$ can be chosen with or without compact resolvent and
the $m$-sectorial operator $A$ constructed above has compact resolvent if and only if
$L$ has by \cite[Thm.\ V.3.40 and VI.3.3]{MR1335452}.
\vspace{-2mm}
\end{proof}

\vskip 0.3cm

\begin{remark}
\label{mistake}
The abstract example of an $m$-accretive operator $T$ with $\dom T\cap\dom T^*$ $=\!\{0\}$  in \cite{MR1885442}
contains a mistake. In fact, the construction $T\!=\!(I\!+\!R)^{-1}(I\!+\!S)^{-1}$ therein relies on the property $\dom R \cap \dom S \!=\! \{0\}$.
However, in the last line of \cite[p.\ 297]{MR1885442} only $(I\!+\!R)^{-1}(I\!+\!R)\!=\!I\uphar \dom R$ holds,
not equality to the identity $I$  on all of $\cH$; as a con\-se\-quence, only the $($trivial$)$ inclusion $\{0\} \!\subset\! \dom R \cap \dom S$
follows and not equality as claimed in \cite{MR1885442}.
\end{remark}

\vskip 0.3cm

\begin{theorem}
\label{subsec:max-sect-dim=n}
For arbitrary $n\in\dN$ there exists an $m$-sectorial operator $A$ in $\sH$ such that $\dim (\dom A\cap\dom A^*)=n$;
moreover, $A$ may be chosen such that $\ker A \!=\! \{0\}$ and with or without compact resolvent.
\end{theorem}

\begin{proof}
Let $L$ be an unbounded selfadjoint operator satisfying \eqref{zero}.
Since the domain of the square of every selfadjoint operator is dense by a theorem also due to von Neumann \cite[Thm.\ V.3.24]{MR1335452},
we can choose linearly independent vectors
\begin{equation}
\label{choice}
  e_1,e_2,\dots,e_n \subset \dom L^2
\end{equation}
and set
\[
H_n:=\lspan\{Le_1, Le_2,\dots, Le_n\}\subset \dom L,\quad \sH_n:=\sH\ominus H_n.
\]
Since $H_n$ is finite dimensional, 
Stenger's lemma \cite{MR0220075}, see also Remark \ref{rem:stenger} below, implies that the operator
\[
\wh L_n:=P_{\sH_n}L\uphar\sH_n,\quad \dom \wh L_n :=\dom L\cap \sH_n,
\]
is a selfadjoint operator in $\sH_n$ with $\ker L_n = \{0\}$.
According to von Neumann's theorem, Theorem~\ref{vnthm}, there exists a bounded selfadjoint operator $\wh G_n$ in $\sH_n$ such that
\begin{equation}
\label{randomn}
  \ker \wh G_n=\{0\},\quad \ran \wh G_n\cap\dom \wh L_n=\{0\}.
\end{equation}
Then
\[
G_n:=\wh G_nP_{\sH_n}
\]
is a bounded selfadjoint operator in $\sH$ with $\ker G_n = H_n$.
If we use the operators $L$ and $G_n$ to define $m$-sectorial operators $A_n$ and $A^*_n$ in $\sH$ in the same way as in the proof of
Theorem~\ref{subsec:max-sect-dim=0}, then \eqref{intersection} implies that
\begin{equation}
\label{intersectionn}
   \dom A_n\cap\dom A_n^*=\big\{ u \in \dom L : u \in \dom L^2, \, G_n L u \in \dom L \big\}
\end{equation}
and $A_n$ has compact resolvent if and only if $L$ has.
The proof is complete if we show that
\begin{equation}
\label{span=}
  \dom A_n\cap\dom A^*_n = \lspan\{e_1,e_2,\dots, e_n\}.
\end{equation}
If $u\!\in\!\dom A_n\cap\dom A^*_n$, then \eqref{intersectionn} and
$\ran G_n\subseteq \ran \wh G_n \subseteq \sH_n$ imply $G_n Lu=\wh G_n P_{\sH_n}Lu \in \dom L\cap\sH_n\!=\!\dom \wh L_n \cap \ran \wh G_n$.
Now \eqref{randomn} yields that $P_{\sH_n}Lu\!=\!0$ or, equivalently, $Lu\!\in\! H_n$,~i.e.\
\[
   Lu=\sum_{k=1}^n c_k Le_k
\]
with $c_k \in \dC$, $k=1,2,\dots,n$. Since $\ker L= \{0\}$, we obtain that $u= \sum_{k=1}^n c_k e_k \in\dom L^2$ and hence `$\subseteq$' in \eqref{span=}.
Conversely, if  $u= \sum_{k=1}^n c_k e_k$ with $c_k \in \dC$, $k=1,2,\dots,n$, then $u\in\dom L^2$ by \eqref{choice} and
$$
  G_nLu=\wh G_nP_{\sH_n}Lu=\wh G_nP_{\sH_n}\left(\sum_{k=1}^n c_k Le_k\right)=0 \in \dom L,
$$
which proves $ u \in\dom A_n\cap\dom A_n^*$ by \eqref{intersectionn} and hence `$\supseteq$' in \eqref{span=}.

The last two claims follow in the same way as in the proof of Theorem~\ref{subsec:max-sect-dim=0}.
\end{proof}


\section{Maximal sectorial operators $\cA$ such that $\codim(\dom \cA\cap\dom \cA^*)=\infty$}
\label{sect.infinite}

In this section we investigate the structure of unbounded $m$-accretive operators with the property that
$\dom T\cap\dom T^*$ is either a finite dimensional subspace or an  infinite dimensional proper closed subspace.

In this case $\codim(\dom T\cap\dom T^*)\!=\!\infty$ and we show that $T$ admits a matrix representation such that the compression of $T$ to
the complement of $\dom T\cap\dom T^*$ is $m$-accretive.

Note that a closed operator for which $\dom T$ contains a closed infinite dimensional subspace, here $\dom T\cap\dom T^*$,
 cannot have compact resolvent by the closed graph theorem and since the resolvents of bounded operators in an infinite dimensional space
cannot be~compact.

\begin{proposition}
\label{ltymhj}
Let $T$ be an unbounded $m$-accretive operator in a Hilbert space~$\sH$ such~that
\[
  \dom T\cap\dom T^* =: \sH_1
\]
is a proper closed subspace of $\sH$. Then $\sH_2\!:=\!\sH\ominus \sH_1$ has infinite dimension, $\dim \sH_2=\infty$,
and with respect to the decomposition $\sH\!=\!\sH_1\oplus\sH_2$ the operator $T$ admits a matrix representation
\begin{equation}
\label{blmatrTT}
   T:=\begin{bmatrix}T_1&K_{12}\cr K_{21}& T_2 \end{bmatrix},
   \quad \dom T=\dom T_1\oplus\dom T_2 =\sH_1\oplus\dom T_2,
\end{equation}
where
\begin{equation}\label{blmatrT2}
\left\{\begin{array}{l}
T_1\;\mbox{is bounded and accretive in }\sH_1,\\
T_2\;\mbox{is $m$-accretive in } \sH_2, \  \dom T_2\cap\dom T^*_2=\{0\},\\
K_{12}\!\in\!\bB(\sH_2,\sH_1), K_{21}\!\in\!\bB(\sH_1,\sH_2).\\
\end{array}
\right.
\end{equation}
\end{proposition}

\begin{proof}
Because $T$ is unbounded and closed in $\sH$ and $\sH_1 \subset \sH$ is a closed subspace, we must have $\dim\sH_2=\infty$. Since $T^*$
is naturally closed in $\sH$, $\sH_1\subset \dom T$, $\sH_{   {1}}\subset\dom T^*$, and $\sH_1$ is a closed subspace, the restrictions
$T\uphar\sH_1$, $T^*\uphar\sH_1$ are closed and everywhere defined, whence bounded operators. Then also the operators
\[
   T_1:=P_{\sH_1}T\uphar\sH_1,\quad K_{21}:= P_{\sH_2}T\uphar\sH_{   {1}},\quad Y:=P_{\sH_2}T^*\uphar\sH_1,
\]
are bounded and $P_{\sH_1}T^*\uphar\sH_1=T^*_1$. Besides, it is not difficult to check that
\[
   \dom T\cap \sH_2=P_{\sH_2}\dom T,\quad \dom T^*\cap \sH_2=P_{\sH_2}\dom T^*.
\]	
Hence {$\dom T\cap \sH_2$ is dense and,} for $f_1\in\sH_1$ and $g_2\in\dom T\cap \sH_2 = P_{\sH_2}\dom T$,
\[
(f_1,Y^*g_2)=(Yf_1,g_2)=(P_{\sH_2}T^*f_1,g_2)=(T^*f_1,g_2)=(f_1, T g_2)=(f_1, P_{\sH_1}T g_2).
\]
This shows that the bounded operator $K_{12}:=Y^* = (T^*\uphar\sH_1)^* P_{\sH_2}$ coincides with
the closure $\overline{P_{\sH_1}T\uphar(\dom T\cap \sH_2)}$ and $T$ has the matrix form \eqref{blmatrTT}
{with $T_2:=P_{\sH_2} T \uphar (\dom T \cap \sH_2)$}.

Clearly, since $T$ is accretive and $W(T_i)\subset W(T)$, $i=1,2$, the operators $T_1$, $T_2$ are accretive as well.
By \eqref{blmatrTT} and because $T_1$, $K_{12}$ and $K_{21}$ are bounded, we have
\[
   T^*=\begin{bmatrix}T^*_1 & K^*_{21}\cr K^*_{12} & T^*_2 \end{bmatrix}, \quad \dom T^*=\dom T^*_1\oplus\dom T^*_2 =\sH_1\oplus\dom T^*_2,
\]
and $T^*_2 = (T \uphar (\dom T \cap \sH_2))^* P_{\sH_2} ({\supseteq} P_{\sH_2}T^*\uphar\sH_2)$.
Since $T$ is $m$-accretive, $T^*$ is accretive. Then $W(T_2^*) \subset W(T^*)$ shows that
$T^*_2$ is accretive and hence $T_2$ is $m$-accretive.
Further,
\[
   \sH_1=\dom T\cap\dom T^*=\left(\sH_1\oplus\dom T_2\right)\cap \left(\sH_1\oplus\dom T^*_2\right)=\sH_1\oplus(\dom T_2{\cap}\dom T^*_2),
\]
which shows that $\dom T_2\cap\dom T^*_2=\{0\}$. 
\end{proof}

\begin{remark}
Note that if the Hilbert space $\sH$ has a decomposition $\sH=\sH_1\oplus\sH_2$ with a closed subspace $\sH_1$ and $\dim\sH_2=\infty$ and if a linear operator $T$ has a matrix representation \eqref{blmatrTT} with a bounded operator $T_1$ in $\sH_1$,
$K_{12}\in\bB(\sH_2,\sH_1)$, $K_{21}\in\bB(\sH_1,\sH_2)$ and a closed densely defined linear operator $T_2$ in $\sH_2$ with
$\dom T_2\cap\dom T^*_2=\{0\}$, then $\dom T\cap\dom T^*=\sH_1$.
\end{remark}

The following additive perturbation result is useful to construct $m$-sectorial coercive operators for which $\dom T\cap\dom T^*$ is a proper closed subspace of $\sH$ with prescribed finite or infinite dimension and for which $\codim(\dom T\cap\dom T^*)\!=\!\infty$.

\begin{proposition}
\label{jdhfnyj}
Let $T_i$ be $m$-sectorial coercive operators in a Hilbert space $\sH_i$, i.e.\ \linebreak $\RE (T_i f_i,f_i)\ge m_i\|f_i\|^2_{\sH_i}$,
$f_i\in\dom T_i$, with some $m_i>0$ for $i=1,2$, and let
$\cK_{12}\in\bB(\sH_2,\sH_1)$, $K_{21}\in\bB(\sH_1,\sH_2)$.
Suppose that $\dom T_2\cap\dom T^*_2=\{0\}$ and
\begin{equation}
\label{eckjdyfr}
  {\frac 12 \big(\|K_{12}]|+\|K_{21}\| \big)} < \min\{m_1,m_2\}.
\end{equation}
Then $\dim\sH_2=\infty$ and the operator $T$ in $\sH=\sH_1 \oplus \sH_2$ given by
\begin{equation}\label{blmatrT}
   T:=\begin{bmatrix}T_1&K_{12}\cr K_{21}& T_2 \end{bmatrix},
   \quad \dom T:=\dom T_1\oplus\dom T_2,
\end{equation}
is $m$-sectorial and coercive with
\[
   {\dim (\dom T\cap\dom T^*)= \dim (\dom T_1\cap\dom T^*_1)}, \quad
   \codim(\dom T\cap\dom T^*)
   =\infty.
\]
If, in addition, the operators $T^{-1}_1$ and $T^{-1}_2$ are compact, then $T^{-1}$ is compact.
\end{proposition}

\begin{proof}
All claims follow from the fact that $T$ is a bounded perturbation of the $m$-sectorial coercive diagonal operator matrix
${\rm diag} \big( T_1, T_2 \big)$ in $\sH$ and that, for $x=(x_1,x_2) \in \dom T = \dom T_1 \oplus \dom T_2$,
\begin{align*}
   \RE (Tx,x) &\ge \min \{ m_1,m_2\} \|x\|^2 - \big(\|K_{12}]|+\|K_{21}\| \big) \|x_1\|\|x_2\|\\
              &\ge \Big( \min \{m_1,m_2\} - \frac 12 \big(\|K_{12}]|+\|K_{21}\| \big) \Big) \|x\|^2.
\qedhere
\end{align*}
\end{proof}

\smallskip

\begin{remark}
i) A sufficient condition for \eqref{eckjdyfr} is $\max\big\{\|K_{12}]|,\|K_{21}\|\big\}< \min\{m_1,m_2\}$.

ii)
If $K_{21}\!=\!-K_{12}^*$, then all claims in Proposition \ref{jdhfnyj} hold without assumption \eqref{eckjdyfr}, i.e.\ without any restriction on the norm of $K_{12}$; in this case, for $x\!=\!(x_1,x_2) \!\in\! \dom T \!=\! \dom T_1 \oplus \dom T_2$,
\[
\RE (Tx,x) = \RE (T_1x_1,x_1) + \RE (T_2x_2,x_2) \ge \min \{m_1,m_2\} \|x\|^2.
\]
\end{remark}


\section{Maximal sectorial and maximal accretive operators with $\codim (\dom T \cap \dom T^*) \in \dN_{0} \cup \{\infty\}$}
\label{sec:extensions-S_z}

In this section we construct unbounded $m$-accretive and even $m$-sectorial operators $T$ such that
$\dom T\cap\dom T^* $ is an infinite dimensional closed subspace and all three possibilities for its complement
are exhausted:
\begin{enumerate}
\item {$\codim\left(\dom T\cap\dom T^*\right)=\infty$,}
\item $\codim\left(\dom T\cap\dom T^*\right)=n \in \dN$,
\item $\codim\left(\dom T\cap\dom T^*\right)=0$;
\end{enumerate}
in the latter case, $\dom T\cap\dom T^*$ is dense, but it will not be a core of $T$. In cases (2) and (3),
the operator $T$ can be arranged to have compact resolvent, while in (1) this is not possible due to the closed graph theorem,
see the beginning of Section \ref{sect.infinite}.

Our main tool is a special type of $m$-accretive extensions of closed densely defined  sectorial operators
which were defined in \cite{MR1772627} and further studied in \cite{MR1670389, MR2240273, MR3696196}.

For a closed densely defined sectorial operator $S$ in a separable infinite dimensional complex  Hilbert space $\sH$ we set
\begin{equation}
\label{orth}
  \sM_{\bz }:=\ran (S-\overline{z}I), \quad \sN_z:=\sH\ominus \sM_{\bz }=\ker (S^*-z I ),\quad z\in \dC,
\end{equation}
and we define a family of linear operators $S_z$ in $\sH$ for $z\in\dC,$ $\RE z\le 0$, by
\begin{equation}
\label{eq:Szdef1}
  \dom S_z:=\dom S\dotplus\sN_z,\quad
  S_zf\!:=\!S{f_S}\!-\!z\varphi_z,\ \ f\!=f_S+\!\varphi_z, \;f_S\!\in\!\dom S, \quad \f_z\!\in\!\sN_z.
\end{equation}
Since $\dom S\cap\sN_z=\{0\}$, the operator $S_z$ is well-defined.
We mention that the Friedrichs and Krein-von Neumann extensions of $S$ are the limits in strong resolvent sense \cite[Sect.~9.3]{MR1887367}
\[
  S_K=\mbox{sr-\!}\lim\limits_{a\nearrow-0}S_a,\quad
  S_F=\mbox{sr-\!\!}\lim\limits_{a\searrow-\infty}S_a
\]%
of the operators {$S_a$, $a\in (-\infty,0)$}, see \cite{MR1670389} and \cite{MR2240273}.
The following properties of the family $S_z$, $z\in\dC,$ $\RE z\le 0$, play a crucial role in the sequel.

\begin{proposition}
\label{prop:ext}
Let $z\in\dC{\setminus\{0\}}$, $\RE z\le 0$. \vspace{1mm} Then
\begin{enumerate}
\newcommand\wtS{{\wt S}}
\item[{\rm i)}]  $S_z$ is an $m$-accretive extension \vspace{1mm} of~$\,S$;
\item[{\rm ii)}] $S^*_z$,
given by $S^*_z\!=\!S^*\uphar\dom S^*_z$ on $\dom S^*_z\!=\!\big\{h\!\in\!\dom S^*\!: (S^*+\bz  I)h\!\in\! \sM_{\bar z} \big\}$, satisfies
\begin{alignat}{2}
\label{sopryazh}
   \dom S^*_z&=\big(2\,\RE z\, I-P_{\sN_z} ({\wt S}^*\!+ \bz  I)\big)\dom {\wt S}^* \quad
	 && \mbox{if }\ \RE z\!<\!0;\\
\label{sopryazh2}
   \dom S^*_{\I x}&=\big(\wt S^*\!-\I x I\big)^{-1}\big(S+\I x I\big)\dom S\dotplus \sN_{\I x} \quad
	 && \mbox{if }\ \RE z\!=\!0, \, z\!=\!\I x,\, x\!\in\!\dR\!\setminus\!\{0\}, \hspace{-8mm}
\vspace{-3mm}
\end{alignat}
where $\wt S$ is any $m$-sectorial extension of \vspace{1mm} $\,S$, e.g.\  its Friedrichs extension~$S_F$;
\item[{\rm iii)}] if $S$ is coercive, i.e.\ $\RE(Sf,f))\ge m||f||^2$, $f\in\dom S$, for some $m>0$, then the operators
$S_z$ for $\RE z < 0$ and  $T_{\I x}:=S_{\I x}+\I x I$ for $x\in\dR\setminus\{0\}$ are $m$-sectorial.
\end{enumerate}
\end{proposition}

\begin{proof}
Since $S$ is sectorial, it admits at least one $m$-sectorial extension $\wt S$, e.g.\ the Friedrichs extension~$S_F$  \cite[Sect.\ VI.2.3]{MR1335452}.
Then ${\wt S}^*$ is also $m$-sectorial \cite[Thm.~VI.2.5]{MR1335452} and therefore $\{\zeta \!\in\! \dC \!\setminus\! \{0\}: \RE \zeta \!<\! 0 \}$ $\subset \rho ({\wt S}^*)$. Therefore, for every $h\in\dom S^*$, there is a unique  $f_{\wt S^*}\in\dom \wt S^*$ such that $(S^*-z I)h=(\wt S^*-z I)f_{\wt S^*}$.
Since $\wt S^*=S^*\uphar\dom \wt S^*$, this implies $(S^*-zI)(h-f_{\wt S^*})=0$ and thus $\f_z:=h-f_{\wt S^*} {\in} \sN_z$.
The other inclusion being obvious, we have thus shown that
\begin{equation}
\label{domS*}
  \dom S^*=\dom \wt S^*\dotplus\sN_z \quad \mbox{if} \quad  \RE z\le 0, \, z\ne 0.
\end{equation}

 i) Let $z\in \dC\setminus \{0\}$, $\RE z \le 0$. Since for $g \in \dom S_z$, $g=f_s + \phi_z$ with $f\in \dom S$, $\phi_z \in \sN_z$,
\begin{align*}
 \RE (S_zg,g) &= \RE \big( (Sf_S,f_S) - z (\phi_z,Sf_S) - \bz (Sf_S,\phi_z) + |z|^2 (\phi_z,\phi_z) \big) \\
 &= \RE  (Sf_S,f_S) +  |z|^2 (\phi_z,\phi_z) \ge 0,
\end{align*}
$S_z$ is accretive. To show that $S_z$ is $m$-accretive, we prove that $S_z^*$ is accretive. Since $S_z^* \subset S^*$,
every $h\in \dom S_z^*$ can be written as $h=h_* + \psi_z$ with $h_* \in \dom \wt S$, $\psi_z \in \sN_z$.
Further, since $\sN_z \subset \dom S_z \cap \dom S^*$ and ${\wt S}^* \subset S^*$, we have
\[
 (S_z^*h,h)= (S^*(h_*+\psi_z),h_*) + (h_*+\psi_z, S_z \psi_z) = (\wt S^*h_*,h_*) + z(\psi_z,h_*) - \bz (h_*,\psi_z) - \bz (\psi_z,\psi_z).
\]
Because $\wt S$ is $m$-accretive, ${\wt S}^*$ is accretive and hence
\[
 \RE (S_z^*h,h) = \RE (\wt S^*h_*,h_*) - \RE z (\psi_z,\psi_z) \ge 0;
\]
for different proofs which require the distinction between $\RE z <0$ and $z\in \I\dR \setminus \{0\}$ comp.\ \cite[p.~4]{MR1772627}
and \cite[Prop.\ 2.7]{MR3696196}, respectively.

ii) Since $S_z^* \subset S^*$, we have $h\in\dom S_z^*$ if and only if $h\in \dom S^*$ and the mapping $\dom S_z=\dom S \dotplus \sN_z \to \dC$, $g=f\!+\!\f_z \mapsto (S_zg,h)$ is continuous with
\[
  (g, S^*h) = (S_z g, h) = (Sf\!-\! z \f_z,h) = (f,S^*h) - (\f_z, \bz h) = (g,S^*h) - (\f_z,(S^*+\bz I)h)
\]
or, equivalently, $(\f_z,(S^*+\bz I)h)=0$. Since $\sN_z^\perp = \ran(S-\bz  I) =\sM_{\bar z}$, this shows that
\begin{equation}
\label{sopryazh1}
  \dom S^*_z 
   = \big\{h\in\dom S^*:  P_{\sN_z} (S^*+\bz  I)h =0 \big\},
\end{equation}
comp.\ \cite{MR3696196} for a different formulation and proof.

To show \eqref{sopryazh}, \eqref{sopryazh2}, let $z\in\dC$, $\RE z \le 0$, $z\ne 0$. By  \eqref{sopryazh1} and \eqref{domS*}, it follows that
$h \in \dom S_z^*$ if and only if $h=h_{{\wt S}^*} + \psi_z$ with $h_{{\wt S}^*}\in\dom \wt S^*$, $\psi_z \in \sN_z$ and
\begin{align}
\label{equivSz*}
  0=P_{\sN_z} (S^*+\bz  I)(h_{{\wt S}^*}+\psi_z)= P_{\sN_z} ({\wt S}^*+\bz  I)h_{{\wt S}^*} + 2 \,\RE z \,\psi_z,
\end{align}
where we have used that ${\wt S}^* \!\subset S^*$, $\sN_z \subset \dom S^*$ and $(S^*\!+\bz I) \psi_z = (z+\bz) \psi_z = 2 \,\RE z \,\psi_z \in \sN_z$.
If $\RE z<0$, then \eqref{equivSz*} is equivalent to $\psi_z= - \frac 1{2\,\RE z} P_{\sN_z} ({\wt S}^*\!+\bz  I)h_{{\wt S}^*}$, which together with $h=h_{{\wt S}^*} + \psi_z$ proves \eqref{sopryazh};
if $\RE z=0$, $z=\I x$, $x\ne 0$, then \eqref{equivSz*} is equivalent to $(\wt S^*\!+\!\bz I) h_{{\wt S}^*}= P_{\sM_z} (\wt S^*\!+\!\bz I) h_{{\wt S}^*} = (S\!-\!\bz I)f$ for some $f\in \dom S$,
which together with $h=h_{{\wt S}^*} + \psi_z$ proves~\eqref{sopryazh2}.

iii) Suppose that the sectorial operator $S$ is coercive. By i), it remains to be shown that
$S_z$ for $\RE z<0$ and $T_{\I x}$ for $x\in \dR\setminus\{0\}$ are sectorial. For the case $\RE z <0$ this was proved
in \cite[Thm.\ 1.1]{MR1772627}.
In the other case, for $g=f_S+\f_{\I x}\in\dom S_{\I x}$ with $f_S\in\dom S$, $\f_{\I x}\in\sN_{\I x}= (\ran (S+ \I x I))^\perp$, we have
\[
  \left(T_{\I x}g,g\right)
	=\big( (S+\I xI)f_S-\I x \f_{\I x}+\I x\f_{\I x},f_S+\f_{\I x} \big)
  =((S+\I xI)f_S,f_S)=(Sf_S,f_S)+\I x||f_S||^2
\]
and hence
\[
  \IM \left(T_{\I x}g,g\right)=\IM (Sf_S, f_S)+x||f_s||^2,\quad
  \RE \left(T_{\I x}g,g\right)=\RE (Sf_S, f_S).
\]
{By the assumptions on $S$, there exist $k$, $m>0$ such that}
\begin{align*}
  |\IM \left(T_{\I x}g,g\right)|
  \!\le\!|\IM (Sf_S, f_S)|+|x|||f_s||^2 \!\le\! \left(k \!+\!\cfrac{|x|}{m}\right)\RE (Sf_S,f_S)
  \!=\!\left(k \!+\!\cfrac{|x|}{m}\right)\RE \left(T_{\I x}g,g\right),
\end{align*}
and thus $T_{\I x}$ is sectorial.
\end{proof}

\begin{remark}
It was shown in \cite[Thm.~1.1]{MR1772627} for the case of an operator and in \cite[Thm.~4]{MR1670389} for the case of a linear relation
that the following are equivalent:
\begin{enumerate}
\def\labelenumi{\rm (\roman{enumi})}
\item for every  $z\!\in\!\dC$, $\RE z\!<\!0$, the operator $S_z$ is $m$-sectorial;
\item for every $z\!\in\!\dC$, $\RE z\!<\!0$, there exists $k(z)\!>\!0$ with $(Sf,f)\!\ge\! k(z)||P_{\sN_z}f||^2\!\!$, $f\!\in\!\dom S$.
\item $\dom S^*\subset \dom {\mathfrak s}_K$,
\end{enumerate}
where ${\mathfrak s}_K$ is the closure of the form associated with the Krein-von Neumann ex\-tension $S_K$ of~$S$.
In particular, $\dom S^*\!\subset\! {\mathfrak s}_K$ if $S$ is coercive, i.e.\ $\RE (Sf,f)\!\ge\! m||f||^2$, $f\!\in\!\dom S$, for some $m\!>\!0$.
\end{remark}

\begin{remark}\label{dffcf}
If $S$ is a non-negative closed densely defined operator, then  \eqref{domS*} holds with
an arbitrary non-negative selfadjoint extension $\wt S$ of $S$ and hence
\begin{alignat}{2}
\label{sopryazh-symm}
   \dom S^*_z&=\big(2\,\RE z\, I-P_{\sN_z} ({\wt S}+\bz  I)\big)\dom {\wt S} \quad
	 && \mbox{if }\ \RE z\!<\!0;\\
\label{sopryazh2-symm}
   \dom S^*_{\I x}&=\big(\wt S-\I x I\big)^{-1}\big(S+\I x I\big)\dom S\dotplus\sN_{\I x} \quad
	 && \mbox{if }\ \RE z\!=\!0, \, z\!=\!\I x,\, x\!\in\!\dR\!\setminus\!\{0\}.
\vspace{-3mm}
\end{alignat}
\end{remark}

\smallskip

The next two theorems, are the main results of this section. For three different cases of $z\in \dC \setminus \{0\}$, $\RE z \le 0$,
we determine the domain intersections $\dom S_z\cap\dom S^*_z$ and their properties when $S$ is a non-negative closed densely defined symmetric operator.

In one of these cases, the property $\dom S^2=\{0\}$ plays a role; note that the existence of such operators $S$ was proved in \cite{MR0003468, MR0003469}, see also Remark \ref{dom0} below.

\begin{theorem}
\label{intrsc}
Let $S$ be a non-negative closed densely defined symmetric operator in $\sH$
and let $S_z$ for $z\in \dC \setminus \{0\}$, $\RE z \le 0$,
be the $m$-accretive extensions of $S$ defined in \eqref{eq:Szdef1}.
\begin{enumerate}
\item[{\rm i)}] If $\,\RE z<0,$  then
\begin{equation}
\label{int11}
   \dom S_z\cap\dom S^*_z= \big( \RE z \, P_{\sM_{\bar z}}+\I \,\IM z \, P_{\sN_z} \big) \,\dom S.
\end{equation}
\item[{\rm ii)}] If $a<0$, then
\begin{equation}\label{int12}
   \dom S_a\cap\dom S^*_a= P_{\sM_{a}}\dom S=\dom S_a\cap\sM_a=\dom S^*_a\cap\sM_a;
\end{equation}
further, $T_a\!:=\!P_{\sM_a}S_a\uphar(\dom S_a\cap\sM_a)$ is selfadjoint and non-negative in the Hilbert space $\sM_a$ and
\vspace{1mm} $T_a\!=\!P_{\sM_a}S^*_a\uphar(\dom S_a\cap\sM_a)$.
\item[{\rm iii)}] If $\,z=\I x$, $x\in\dR\setminus\{0\}$, then
\begin{equation}
\label{int22}
   \dom S_{\I x}\cap\dom S^*_{\I x}
   =\left((S+\I xI)\dom S^2\right)\oplus\sN_{\I x}
   =\big(\sM_{-\I x}\cap\dom S\big)\oplus\sN_{\I x}; \hspace{-8mm}
\end{equation}
{
in \vspace{-1mm} particular,
\begin{equation}\label{int23}
\dom S_{\I x}\cap\dom S^*_{\I x}=\sN_{\I x} \iff \dom S^2=\{0\}.
\end{equation}
}
\end{enumerate}
\vspace{-4mm}
\end{theorem}

\begin{proof}
Throughout this proof let $\wt S$ be some non-negative selfadjoint extension of \vspace{1mm}$S$.

i) Let $\RE z<0$ and $h\in\dom S_z\cap\dom S^*_z$. Then, due to \eqref{eq:Szdef1} and \eqref{sopryazh-symm}, we have
\[
   h=g_S+\f_z=2\,\RE z \,f_{\wt S}-P_{\sN_z}(\wt S+\bar zI)f_{\wt S}
\]
with $g_S\in\dom S$, $f_z\in\sN_z$ and $f_{\wt S}\in\dom \wt S$ and hence
\[
   \dom \wt S \ \ni \ 2\,\RE z \,f_{\wt S}-g_S=P_{\sN_z}(\wt S+\bar zI)f_{\wt S}+\f_z\ \in \ \sN_z.
\]
By \eqref{domS*} we have $\dom \wt S\cap\sN_z=\{0\}$ and so we obtain
\[
   f_{\wt S}=\cfrac{1}{2\,\RE z} \,g_S \in\dom S ,\quad  (\wt S+\bar z I)f_{\wt S}=\cfrac{1}{2\,\RE z} (S+\bar z I)g_S,
\]
and, using that $P_{\sN_z} (S-\bz I)=0$ by \eqref{orth},
\begin{align*}
   h & = g_S-\cfrac{1}{2\,\RE z} P_{\sN_z}(S+\bar z I)g_S
       = g_S-\cfrac{1}{2\,\RE z} P_{\sN_z}(S-\bar z I) g_S-\cfrac{\bar z}{\RE z} P_{\sN_z}g_S\\
     & =\cfrac{1}{\RE z}\big(\RE z \, P_{\sM_{\bar z}}g_S\!+\!\RE z\, P_{\sN_z}g_S\!-\!(\RE z\!-\!\I \,\IM z)P_{\sN_z}g_S\big)
     =\big(\RE z\,P_{\sM_{\bar z}}\!+\!\I\,\IM z\,P_{\sN_z}\big)\cfrac{1}{\RE z}g_S.
\end{align*}
Thus \eqref{int11} holds.

ii) Let $a<0$. The first equality in \eqref{int12}
is immediate from the identity \eqref{int11}.

The in\-clusions $P_{\sM_{a}}\dom S \subset\dom S_a\cap\sM_a$ and $P_{\sM_{a}}\dom S \subset\dom S_a^*\cap\sM_a$ are obvious
{from the first identity in \eqref{int12}}.

Vice versa, let first $g=f_S+\f_a\in\dom S_a\cap \sM_a$ with $f_S\in\dom S$, $\f_a\in\sN_a$. Then
$0=P_{\sN_a}g=P_{\sN_a}f_S -\f_a$ and hence $g=P_{\sM_a}f_S\in P_{\sM_a}\dom S$. Thus the second equality in~\eqref{int12} is proved.

Secondly, let $g\in\dom S^*_a\cap \sM_a$. Then, by \eqref{sopryazh-symm},
\begin{align*}
   g=2 af_{\wt S}-P_{\sN_a}(\wt S+aI)f_{\wt S}
   =2aP_{\sM_a}f_{\wt S}-P_{\sN_a}(\wt S-a I)f_{\wt S}
\end{align*}
with $f_{\wt S} \in\dom \wt S$ and $0=P_{\sN_a}g=P_{\sN_a}(\wt S\!-\!a I)f_{\wt S}$.
The latter implies that $(\wt S\!-\!a I)f_{\wt S} \in \sM_a$, $(\wt S\!-\!a)f_{\wt S} = (S\!-\!a) f$ with $f\!\in\! \dom S$.
Since $S\!\subset\! \wt S$ and $\wt S \!-\!a$ is bijective, it follows that
$f_{\wt S} =f \in\dom S$. Hence $g = P_{\sM_a} (2a f ) \in P_{\sM_a}\dom S$.
This completes the proof of the third equality in \eqref{int12}.

To show that the operator $T_a:=P_{\sM_a}S_a\uphar(\dom S_a\cap\sM_a)$ is selfadjoint we note that,
by \eqref{int11},  for $h\in \dom S_a\cap\sM_a = P_{\sM_{a}}\dom S$,
$h=P_{\sM_a}f_S$ with $f_S\in \dom S$,
\[
   T_a h=P_{\sM_a}S_aP_{\sM_a}f_S=P_{\sM_a}S_a(f_S-P_{\sN_a}f_S)=P_{\sM_a}Sf_S
\]
and hence, by the definition of $\sN_a \subset \dom S^*$ in \eqref{orth},
\begin{align*}
   (T_ah,h)&=(P_{\sM_a}Sf_S, P_{\sM_a}f_S)=(Sf_S, P_{\sM_a}f_S)=(Sf_S,f_S- P_{\sN_a}f_S)\\
           &=(Sf_S,f_S)-(f_S, S^*P_{\sN_a}f_S)=(Sf_S,f_S)-(f_S, aP_{\sN_a}f_S)\\
					 &=(Sf_S,f_S)-a||P_{\sN_a}f_S||^2\ge 0.
\end{align*}
Therefore $T_a$ is symmetric and non-negative. Further,
\[
   (T_a-aI)h=P_{\sM_a}Sf_S-aP_{\sM_a}f_S=P_{\sM_a}(S-a I)f_S=(S-a I)f_S
\]
shows $\ran (T_a-aI)=\sM_a$ which yields that $T_a$ is selfadjoint in $\sM_ a$.
In order to prove that $T_a=P_{\sM_a}S^*_a\uphar(\dom S^*_a\cap\sM_a)=:Q_a$,
we note that $S$ is symmetric and, by \eqref{int11},  for $h\in \dom S_a^*\cap\sM_a = P_{\sM_{a}}\dom S$,
$h=P_{\sM_a}f_S$ with $f_S\in \dom S \subset \dom S^*$,
\[
   Q_a P_{\sM_a}f_S=P_{\sM_a}S^*_a(f_S-P_{\sN_a}f_S)=P_{\sM_a}S^*(f_S-P_{\sN_a}f_S)=P_{\sM_a}Sf_S=T_aP_{\sM_a}f_S.
\]

iii) The last equality $\sM_{-\I x}\cap\dom S=(S+\I x I)\dom S^2$ in \eqref{int22} is clear. To prove the first,
suppose that $h\in\dom S_{\I x}\cap\dom S^*_{\I x}$. Then, by \eqref{eq:Szdef1},
\eqref{sopryazh2-symm} and since $S\subset \wt S$, we~have
\begin{equation}
\label{g_S}
   h=g_S+\f_{\I x}= (\wt S-\I x I)^{-1}(S+\I x I)f_S+\psi_{\I x} = f_S+2\I x(\wt S-\I x I)^{-1}f_S+\psi_{\I x}
\end{equation}
with $g_S$, $f_S\in\dom S$ and $\f_{\I x}$, $\psi_{\I x}\in\sN_{\I x}$. Hence
\[
\dom \wt S \ \ni \ 2\I x(\wt S-\I x I)^{-1}f_S-g_S+f_S=\f_{\I x}-\psi_{\I x} \ \in \ \sN_{\I x}.
\]
Since $\dom \wt S\cap\sN_{\I x}=\{0\}$ by \eqref{domS*}, we obtain
$\f_{\I x}-\psi_{\I x}=0$ and
\begin{equation}
(S-\I xI)(g_S-f_S) ={(\wt S-\I xI)(g_S-f_S)}= 2\I x f_S \in \dom S.
\end{equation}
The latter implies that $g_s\!-\!f_s \in \dom S^2$.  By \eqref{g_S} we have
$(S \!-\! \I x I) g_S = (S\!+\! \I x I) f_S$ and~so
\[
   g_S = \frac 1{2 \I x} (S+\I xI) (f_S-g_S) \in (S+\I x I) \dom S^2,
\]
which proves the inclusion ``$\subset$'' in the first equality in \eqref{int22}.
Conversely, for $h_S\!\in\! \dom S^2$~set
\begin{align*}
  g_S := \frac 1{2\I x} (S+\I x I) h_S \in \dom S \subset \dom S_{\I x}, \quad
  f_S := \frac 1{2\I x} (S-\I x I) h_S \in \dom S.
\end{align*}
Then  $g_S = f_S+2\I x(\wt S-\I x I)^{-1}f_S \in \dom S_{\I x}^*$ which proves $(S+\I x I) \dom S^2 \!\subset  \dom S_{\I x} \cap  \dom S_{\I x}^*$.
Clearly, $\sN_{\I x}\subset\dom S_{\I x}\cap\dom S^*_{\I x}$ by definition \eqref{orth} and \eqref{sopryazh2}.
Hence the proof of the inclusion ``$\supset$'' in the first equality in \eqref{int22} is complete.
\end{proof}

\begin{remark}
\label{rem:stenger}
W.~Stenger \cite{MR0220075} proved that for an unbounded selfadjoint operator $A$ in a Hilbert space $\sH$ and a finite codimensional $($hence closed$)$ subspace $\cF \subset \sH$ the so-called \emph{compression} $P_{\cF}A\uphar(\dom A\cap\cF)$ is selfadjoint in $\cF$. This result {was}
extended by M.A.~Nudelman \cite{MR2806469} to $m$-dissipative operators, i.e.\ if $A$ is $m$-dissipative in $\sH$, then
$P_{\cF}A\uphar(\dom A\cap\cF)$ is $m$-dissipative in~$\cF$.

Theorem {\rm \ref{intrsc}}, when applied with a non-negative symmetric $S$ having infinite deficiency indices, provides
examples of $m$-accretive operators $S_a$, $a<0$, for which even the compressions $P_{\sM_a}S_a\uphar(\dom S_a\cap\sM_a)$
to the infinite codimensional subspaces $\sM_a$ are $m$-accretive, and even non-negative selfadjoint, in $\sM_a$.
\end{remark}

\begin{theorem}
Let $S$ be a closed densely defined non-negative symmetric, but not selfadjoint operator in $\sH$
with associated defect subspace $\sM_{\bar z}$ and $\sN_z \subset \sH$  as in \eqref{orth},
and let $S_z$ for $z\in \dC \setminus \{0\}$, $\RE z \le 0$, be the $m$-accretive extension of $S$ in~\eqref{eq:Szdef1}.
\begin{enumerate}
\item[{\rm i)}]
If $\,\RE z\!<\!0$ and $\IM z\!\ne\! 0$, then
\[
   \dom S_z\cap\dom S^*_z \ \mbox{is dense in} \ \sH, \ \mbox{but neither a core of $\,S_z$ nor of $\,S^*_z$}.
\]
\item[{\rm ii)}] If $a<0$, then $\left(\dom S_a\cap\dom S^*_a\right)^\perp =\sN_a$ and hence
\[
  \codim\left(\dom S_a\cap\dom S^*_a\right)=\dim\sN_a\le\infty.
\]
\item[{\rm iii)}] If $x\in\dR\setminus\{0\}$, \vspace{1.75mm} then
\begin{enumerate}
\item[{\rm a)}]{if} $\,\dom S_{\I x}\cap\dom S^*_{\I x}$ is dense in $\sH$ and thus, in particular, if {$\,\dim \sN_{\I x}\!<\!\infty$},~then
\[\mbox{
$\dom S_{\I x}\cap\dom S^*_{\I x}$ is neither a core of $\,S_{\I x}$ nor of \vspace{1mm} $\,S^*_{\I x}$,
}\]
\item[{b)}] if $\dom S^2=\{0\}$, then $\dim \sN_{\I x}=\infty$ and
therefore
\[
   \dim \big(\dom S_{\I x}\cap\dom S^*_{\I x} \big)=\infty, \quad
   \codim \big(\dom S_{\I x}\cap\dom S^*_{\I x}\big)=\infty.
\hspace{-8mm}
\]
\end{enumerate}
\end{enumerate}
\end{theorem}

\begin{proof}
i)
Assume $\RE z<0$ and $\IM z\ne 0$.
First we show that $\dom S_z \cap \dom S_z^*$ is 1) neither equal to $\dom S_z$, 2) nor to $\dom S_z^*$. For 1), let $h\in \sN_z \setminus \{0\}$.
Then $h \in \dom S_z\setminus \{0\} $, but $h \notin \dom S_z^*$. Otherwise, by \eqref{int11}, $h = (\RE z \,P_{\sM_{\bz}} + \I \,\IM z \, P_{\sN_z} ) f$
for some $f\in \dom S$. This implies $P_{\sM_{\bz}} f=0$, $\I \,\IM P_{\sN_{\bz}}f = h$ and hence $h= \I \,\IM z\, f \in \dom S \cap \sN_z =\{0\}$, a contra\-diction. For 2), since $S$ is not selfadjoint, we can choose $h \in (2 \,\RE z \,I - P_{\sN_{\bz}} (\wt S + \bz I))f$ with $f \in \dom \wt S \setminus \dom S$. Then, by \eqref{int12}, $ h \in \dom S_z^* \setminus \{0\}$, but $h \notin \dom S_z$. Otherwise, by~\eqref{eq:Szdef1}, $h = f_S +\f_z$ with $f_S\in \dom S$,
$\f_z\in \sN_z$. This implies
\[
 \dom \wt S \ \ni \ 2 \, \RE z \, f - f_S = \f_z + P_{\sN_{\bz}} (\wt S + \bz I) f \ \in \sN_z,
\]
and hence, since $\dom \wt S \cap \sN_z = \{0\}$ by \eqref{domS*},  $ f = \frac 1{2 \,\RE z} f_S \in \dom S$, a contradiction.

Secondly, we prove that $\dom S_z \cap \dom S_z^*$ is dense in $\sH$. Let $h\in \sH$ be arbitrary. Since $\dom S$ is dense in $\sH$, there exists
a sequence $\{f_n\}_{n\in\dN} \subset \dom S$ such that $f_n \to \frac 1{\RE z} P_{\sM_{\bz}}h + \frac 1 {\I \,\IM z} P_{\sN_z}h$, $n\to\infty$. Then
$g_n:=(\RE z \,P_{\sM_{\bz}} + \I \,\IM z \,P_{\sN_z} ) f_n \in \dom S_z \cap \dom S_z^*$, $n\in\dN$, by \eqref{int11} and $g_n \to P_{\sM_{\bz}} h +
P_{\sN_z}  h = h$, $n\to\infty$.

To prove that $\dom S_z \cap \dom S_z^*$ is neither a core of $S_z$ nor of $S^*_z$,
let $\|\cdot\|_{S_z}$ and $\|\cdot\|_{S^*_z}$ be the graph norms on $\dom S_z$ and $\dom S^*_z$, respectively. Assume $\{f_n\}_{n\in\dN}\subset\dom S_z\cap\dom S^*_z$ is a Cauchy sequence with respect to $\|\cdot\|_{S_z}$, i.e.\ $\{f_n\}_{n\in\dN}$ and $\{S_zf_n\}_{n\in\dN}$ are Cauchy sequences
in $\sH$. Then, by \eqref{int11} and \eqref{eq:Szdef1}, there exist $f^{(n)}_S \in \dom S$, $n\in\dN$, such that
\begin{align*}
f_n&=\big(\RE z\,P_{\sM_{\bar z}}\!+\!\I\,\IM z \,P_{\sN_z}\big)f^{(n)}_S=\RE z \,f^{(n)}_S\!-\!\bz P_{\sN_z}f^{(n)}_S\!,\quad
S_zf_n=\RE z \,Sf_S^{(n)}\!+\!|z|^2 P_{\sN_z}f^{(n)}_S\!.
\end{align*}
It follows that $\{P_{\sM_{\bar z}}f_S^{(n)}\}_{n\in\dN}$, $\{P_{\sN_z}f_S^{(n)}\}_{n\in\dN}$, and hence also $\{f_S^{(n)}\}_{n\in\dN}$ and $\{Sf_{S}^{(n)}\}_{n\in\dN}$ are Cauchy sequences in $\sH$. Since $S$ and $S_z$ are closed, we conclude that the limit of $\{f_S^{(n)}\}_{n\in\dN}$ lies in $\dom S$ and the limit of $\{f_n\}$ 
lies in $\dom S_z\cap\dom S^*_z$. This proves that $S_z \uphar (\dom S_z\cap\dom S_z^*)$ is closed.
Because $S^*_z$ is a closed restriction of $S^*$, similar arguments show that also $S_z^* \uphar (\dom S_z\cap\dom S_z^*)$ is closed.

Altogether, it follows that $\dom S_z\cap\dom S^*_z$ is neither a core of $S_z$ nor of $S^*_z$.

ii) The claim is immediate from the first identity in \eqref{int12} since $\dom S$ is dense.

iii)  By \eqref{int22}, we have
\[
   \left(\dom S_{\I x}\cap\dom S^*_{\I x}\right)^\perp=\sM_{-\I x}\ominus\left((S+\I x I)\dom S^2\right)=\sM_{-\I x}\ominus\left(\sM_{-\I x}\cap\dom S\right).
\]
a)
Suppose $\dom S_{\I x}\cap\dom S^*_{\I x}$ is dense in $\sH$. Due to \eqref{int22} this is equivalent to $(S+\I xI)\dom S^2$
being dense in $\sM_{-\I x}$. Using \eqref{eq:Szdef1}, the inclusion $(S+\I xI)\dom S^2\!\subset\!\dom S$ and the symmetry of~$S$, we obtain that
\begin{align*}
  S_{\I x}(S+\I x I)f_S&
	=   S(S+\I x I)f_S =
  S^*_{\I x}((S+\I xI)f_S),
	\quad f_S\in\dom S^2.
\end{align*}
Hence \eqref{int22} yields that the operators $S_{\I x}\uphar(\dom S_{\I x}\cap\dom S^*_{\I x})$ and $S^*_{\I x}\uphar(\dom S_{\I x}\cap\dom S^*_{\I x})$ are closed. Since $S$ is closed, non-negative and unbounded,  $\dom S^2 \subsetneq \dom S$ by \cite[Thm.\ 2.1]{MR760619} and hence $(S+\I xI)\dom S^2 \subsetneq (S+\I xI)\dom S = \sM_{-\I x}$. Together with the closedness of $S_{\I x}$ and $S_{\I x}^*$, it follows that $\dom S_{\I x}\cap\dom S^*_{\I x}$ is neither a core of $S_{\I x}$ nor of $S_{\I x}^*$.

b)
Suppose that $\dom S^2 = \{0\}$. We only have to prove that $\dim \sN_{\I x}=\infty$.
If  $\sN_{\I x}$ were finite dimensional, then $\dom S \cap \ran (S+\I x I) = \dom S \cap \sN_{\I x}^\perp$ is dense in
$\sH$ by \cite[Lemma 2.1]{MR0113146}. In particular, there exists $g\in (\dom S \cap \ran (S+\I x I)) \setminus \{0\}$
and hence $g=(S+\I x I)f$ with some $f \in \dom S \setminus \{0\}$. The latter implies that $Sf=g-\I x f \in \dom S$ and so $f\in \dom S^2 = \{0\}$,
a contradiction.
\end{proof}

\begin{remark}
\label{dom0}
The first example of a closed densely defined symmetric operator $S$ whose square has trivial domain, {$\dom S^2=\{0\}$}, was constructed in \cite{MR0003468, MR0003469}. In \cite{MR712639} an example of a closed densely defined semi-bounded symmetric operator {$S$} with $\dom S^2=\{0\}$ was given.

Further, in \cite[Thm.~5.2]{MR695940} it was shown that, for every unbounded selfadjoint operator $A$ in $\sH$, there exist
closed densely defined $($symmetric$)$ restrictions $A_1$ and $A_2$ of $A$ such that
\[
\dom A_1\cap\dom A_2=\{0\} \quad \mbox{and} \quad \dom A^2_1=\dom A^2_2=\{0\}.
\]
\end{remark}

\begin{remark}
\label{B_x}
The construction in this section yields another class of $m$-sectorial or $m$-accretive operators $T$ with $\dom T \cap \dom T^* = \{0\}$,
different from the one in Section \ref{sec:max-sect-dim=n}. It relies on the existence of a closed densely defined non-negative symmetric operator $S$ with $\dom S^2 = \{0\}$ rather~than on von Neumann's theorem, Theorem \ref{vnthm}, as in the construction in Section~{\rm\ref{sec:max-sect-dim=n}}.
In fact, by \eqref{eq:Szdef1}
\begin{align*}
    &\dom S_{\I x}\cap \sM_{-\I x}=P_{\sM_{-\I x}}\dom S,\\
    & S_{\I x}(P_{\sM_{-\I x}}f_S)\!=\!S_{\I x}(f_S\!-\!P_{\sN_{\I x}}f_S)\!=\!Sf_S\!+\!\I xP_{\sN_{\I x}}f_S\!
      =\!(S\!+\!\I xI)f_S\!-\!\I x P_{\sM_{-\I x}} f_S    {\in\!\sM_{-\I x}} 
\end{align*}
for $f_S\!\in\!\dom S$ and $\sN_{\I x}$ is an eigenspace of the operator $S_{\I x}$ corresponding to the eigenvalue $-\I x$.
 This shows that both $\sM_{-\I x}$ and $\sM_{-\I x}^\perp = \sN_{\I x}$ are invariant for $S_{\I x}$, i.e.\
$\sM_{-\I x}$ is a reducing subspace for $S_{\I x}$ and hence also for $S^*_{\I x}$ \cite[Sect.~2.5]{MR1887367}.
Further, $S_{\I x}\uphar(\dom S_{\I x}\cap\sM_{-\I x})$ and $S^*_{\I x}\uphar(\dom S^*_{\I x}\cap\sM_{-\I x})$ are $m$-accretive in the
Hilbert space $\sM_{-\I x}$.
Moreover, since $\sN_{\I x} = \ker(S^*_{\I x}\!-\!\I x I) \!\subset\!\dom S^*_{\I x}$, we have
\[
   \dom S^*_{\I x}\cap\sM_{-\I x}=P_{\sM_{-\I x}}\dom S^*_{\I x}.
\]
Therefore, the operators $S_{\I x}$ and $S_{\I x}^*$ decompose as
$S_{\I x} = {\rm diag}\,(B_{\I x}, -\I x)$, $S_{\I x}^* = {\rm diag}\,(B_{\I x}^*, \I x)$ in $\sH = \sM_{-\I x} \oplus \sN_{\I x}$
with the $m$-accretive operators
\begin{equation}\label{opbx}
 B_{   {\I}x}:=S_{\I x}\uphar(\dom S_{\I x}\cap\sM_{-\I x}), \quad
    {B_{\I x}^*=S_{\I x}^*\uphar(\dom S_{\I x}^*\cap\sM_{-\I x})}.
\end{equation}
Thus, by Proposition {\rm \ref{prop:ext} iii)} and \eqref{int23}, if $\dom S^2\!=\!\{0\}$, then $T\!=\!B_{\I x}$ is
$m$-accretive in $\sM_{-\I x}$ and $T\!=\!B_{\I x} \!+\! \I x$ is $m$-sectorial in $\sM_{-\I x}$, respectively, with $\dom T \cap \dom T^* \!=\! \{0\}$.
\end{remark}

In the following, for arbitrary $n\in \dN$, we construct a closed densely defined positive definite symmetric operator $S$
with deficiency index $(n,n)$ such that its inverse $S^{-1}$,  defined on $\dom S^{-1}=\ran S$, is compact.

\begin{remark}
\label{dgjkyt}
Let $G$ be a non-negative compact operator in a Hilbert space $\sH$ with $\ker G\!=\!\{0\}$
and define
\[
    A:=G^{-1}, \quad \sH_+:=\dom A=\ran G,\quad (f,g)_{\sH_+}:=(Af,g)=(G^{-1}f, g).
\]
Then $\sH_+$ is a Hilbert space with respect to the scalar product $(\cdot,\cdot)_{\sH_+}$ induced by the norm $\|A^{1/2}\cdot\|$, and
$A$ is a positive definite unbounded selfadjoint operator in $\sH_+$.
Let
\[
   \sH_+\subset \sH\subset \sH_-
\]
be the corresponding rigged Hilbert space, see e.g.~\cite{MR0222718}, \cite{MR3496031},  i.e.\ $\sH_-$ is the closure of $\sH$ with respect to
the norm $\|A^{-1/2}\cdot\|$.

Given $n\in \dN$, we now choose an arbitrary $n$-dimensional subspace $\sL$ in $\sH_-$ such that $\sL\cap \sH=\{0\}$;
note that this is possible since $A$ is unbounded and hence $\sH \subsetneq \sH_-$ is only dense and not closed in $\sH_-$.
If we define the operator $S$ in $\sH_+$ by
\[
    \dom S:=\left\{f\in\sH_+: (f,\varphi)=0,\ \varphi\in\sL\right\}  \subset \dom A,  \quad S:=A\uphar \dom S,
\]
where $(f,\varphi) \!=\! (A^{1/2}f, A^{-1/2}\varphi)$ denotes the canonical pairing,
then $S$ is densely defined posi\-tive definite and symmetric with deficiency indices $(n,n)$, see~\cite{MR3496031}, and $A$ is a
positive definite selfadjoint extension of $S$ in $\cH_+$ with compact resolvent. Hence,  for $\lambda\!\in\!\dC\setminus\dR_+$, the
inverse $(S\!-\!\lambda I)^{-1}\!=\!(A\!-\!\lambda I)^{-1} \uphar \ran (S\!-\!\lambda I)$ on $\dom (S\!-\!\lambda I)^{-1} \!=\! \ran (S\!-\!\lambda I)$ is compact
in~$\sH_+$.
\end{remark}

\begin{remark}
Given $n\in \dN \cup \{\infty\}$, we can choose a  closed densely defined non-negative symmetric operator $S$ in a Hilbert space $\sH$ with
deficiency index $(n,n)$, i.e.\
$\dim \sN_z\!=\!n\le\infty$, $z\in\dC\setminus\dR_+$. Then,
by Theorem \ref{intrsc} ii), the extensions $S_a$, $a<0$, are a family of
$m$-accretive operators such that
\[
   \codim \left(\dom S_a\cap\dom S^*_a\right)=n \in \dN \cup \{\infty\}.
\]

For finite $n\in \dN$, we can even choose a positive definite $S$ with
compact inverse $S^{-1}$ defined on $\dom S^{-1}=\ran S$, see Remark \ref{dgjkyt}.
Due to \eqref{orth}, \eqref{eq:Szdef1}, we have
\[
   (S_z-\bar z I)^{-1}=(S-\bar z I)^{-1}P_{\sM_{\bar z}}-\cfrac{1}{2\,\RE z}P_{\sN_z}, \quad z \in \dC, \ \RE z < 0.
\]
Together with the Hilbert identity for the resolvent, the inverse $(S_z-\lambda I)^{-1}$, which is defined on
$   {\dom (S-\lambda I)^{-1}}\!=\! \ran (S-\lambda I)$,    {$\lambda \!\in\! \dC$, $\RE \lambda \!<\!0$,} is compact.
Hence, in this case, by Proposition~\ref{prop:ext} and Theorem~\ref{intrsc}~ii), the extensions $S_a$, $a<0$, are a family
of $m$-sectorial operators with compact resolvent such that $\codim \left(\dom S_a\cap\dom S^*_a\right)=n \in \dN$.

Note that, if the deficiency numbers of $S$ are infinite, i.e.\ $n=\infty$, then the resolvent of the operator $S_z$ cannot be compact
since the defect subspace $\sN_z$ is an infinite dimensional eigenspace of $S_z$ for the eigenvalue $-z$, see \eqref{eq:Szdef1}.
\end{remark}

\section{
Singular perturbations $T$ of a selfadjoint operator with $\codim\!\big(\dom T\cap \dom T^*\big)\!\in\!\dN$}
\label{section:singperturb}

In this section, we construct unbounded non-selfadjoint densely defined operators $T$ such that
$\codim(\dom T\cap\dom T^*)=n$ for arbitrary given $n\in \dN$ with non-emtpy resolvent set and full numerical range,
$\rho(T)\ne \emptyset$ and $W(T)=\dC$.

The  operators $T$ will be obtained as singular perturbations of an unbounded selfadjoint operator $A$  in a separable
infinite dimensional complex Hilbert space $\sH$ with inner product $(\cdot,\cdot)$. To this end,
we consider the Hilbert space $\sH_+$ induced by the graph norm $\|\cdot\|_+$ of $A$, i.e.\ $\sH_+:=\dom A$ with inner product $(\cdot,\cdot)_+$ given by
\[
   (f,g)_+:=(f,g)+(Af,Ag), \quad f,g\in \sH_+.
\]
Let
\[
   \sH_+\subset \sH\subset \sH_-
\]
be the corresponding rigged Hilbert space, see e.g.~\cite{MR0222718}, \cite{MR3496031},
where $\sH_-$ is the closure of $H$ with respect to the norm
\[
   \|h\|_-:=\sup\limits_{\varphi\in \sH_+,\|\varphi\|_{   {+}}=1}|(\varphi, h)|,
\]
i.e., $\sH_-$ is the Hilbert space of all linear functionals on $\sH_+$ bounded with respect to $\|\cdot\|_+$ 
each of which is determined by a vector $h\in \sH_-$ via
\begin{equation}
\label{eq:pairing}
   l_h(\varphi)=(\varphi, h),\quad \varphi\in \sH_+.
\end{equation}

Clearly,  $ \|f\| \!\le\! \|f\|_+$ and $\|Af\| \!\le\! \|f\|_+$,  $f\in \sH_+$, so the operator $A$ is bounded when act\-ing from $\sH_+$ to $\sH$,
$A\!\in\!\bB(\sH_+,\sH )$. Let $A^\times\!\in\!\bB(\sH, \sH_-)$ be the corresponding adjoint operator,~i.e.,
\[
   (Af,g)=(f, A^\times g)_{{+,-}}, \quad f\in \sH_+,\ g\in \sH,
\]
where $(\cdot,\cdot)_{+,-}$ on the right hand side denotes the canonical pairing between $\sH_+$ and $\sH_-$, see~\eqref{eq:pairing}.
Since $A$ is selfadjoint in $\sH$, the adjoint operator $A^\times\in\bB(\sH, \sH_-)$ is the continuation of $A$ onto $\sH$, {and we therefore} denote $\bA:=A^\times$.

The equalities
\[
||(A+\I I_\sH )f||^2=||(A-\I I_\sH )f||^2=||Af||^2+||f||^2=||f||^2_+, \quad f\in \sH_+,
\]
imply that the operators $A\pm \I I_{\sH}$  map $\sH_+$ onto $\sH$ unitarily and thus their adjoints $\bA\mp \I I_{\sH}$ map $\sH$ onto $\sH_-$ unitarily.

\begin{theorem}
Let $A$ be an unbounded selfadjoint operator in a separable Hilbert space $\sH$ with corresponding rigged Hilbert space $\sH_+\subset \sH\subset \sH_-$.
Suppose $Z\in \bB(\sH,\sH_-)$ satisfies
\begin{equation}
\label{final}
  \|Z\|_{\sH\to \sH_-}<1, \quad \dim\ran Z=n\in\dN,\quad   \ran Z\cap \sH=\{0\}.
\end{equation}
Then the operator {$T$ in $\sH$} given by
\begin{equation}
\label{eq:singpert}
   T:=A+Z^\times, \quad \dom T:=\dom A,
\end{equation}
is closed, $\rho(T)\cap\dC_{{\pm}}\ne\emptyset$, 
its adjoint has the form
\begin{equation}
\label{eq:singpert*}
   T^*=\left(\bA+Z\right)\uphar\left(I_{\sH}+(\bA+ \I I_\sH )^{-1}Z\right)^{-1}\sH_+,
\end{equation}
and
\[
   \codim\left(\dom T\cap \dom T^*\right)=n.
\]
Moreover, if $A$ is chosen to be non-semibounded, ,
then $W(T)=\dC$; if $A$ is chosen to have compact resolvent, then $T$ has compact resolvent.
\end{theorem}

\begin{proof}
Since $Z\in\bB(\sH, \sH_-)$, we have $Z^\times\in\bB(\sH_+,\sH )$.
The assumption $\|Z\|_{\sH\to \sH_-}\!\!<\!1$ in \eqref{final} implies that $b\!:=\!\|Z^\times\|_{\sH_+\to \sH}\!<\!1$ which means~that
\[
  \| Z^\times h \|^2 \le b^2 \| h\|_+^2 = b^2 \| h\|^2 + b^2 \| A h \|^2, \quad h \in \sH_+,
\]
i.e.\ $Z^\times$ viewed as an operator in $\sH$ with domain $\dom Z^\times = \sH_+$ is $A$-bounded with $A$-bound $\le b <1$. Hence,
by \cite[Thm.~VI.1.1]{MR1335452}, the operator $T$ is closed in $\sH$ and thus $\dom T^*$ is dense in $\sH$.
By \cite[Thm.~2.1 i)]{MR3488056}, the spectrum of $T$ lies in some hyperbolic region around the real axis, and so
 $\rho(T)\cap \{ z\in\dC: \IM z \gtrless 0 \} \ne\emptyset$.

To study the adjoint $T^*$, we define the operator
\[
   Q:=\bA+Z \in\bB(\sH, \sH_-).
\]
Then, for $h\in \sH$, we have the equivalences
\begin{align}
\nonumber
   Qh
   \in \sH & \iff \bA h {\pm} \I h+Zh\in \sH \\
\label{form1}
   &\iff (\bA {\pm} \I I_\sH )\left(I_H+(\bA {\pm} \I I_\sH )^{-1}Z\right) h\in \sH \\
   &\iff \left(I_{\sH} +(\bA {\pm} \I I_\sH )^{-1}Z\right) h\in \sH_+.
\nonumber
\end{align}

Observe that, since $||Z||_{\sH\to \sH_-}<1$ by assumption \eqref{final},
\[
   I_{\sH}+(\bA\pm \I I_\sH )^{-1}Z\in \bB(\sH ), \quad {\big(I_{\sH}+(\bA\pm \I I_\sH )^{-1}Z \big)^{-1} \in \bB(\sH)}.
\]

The operator $T$ in \eqref{eq:singpert} is the operator 
\[
   Q^\times=A+Z^\times \in \bB(\sH_+, \sH)
\]
viewed as an operator acting in $\sH$ since
\[
  \{ h \in \sH : Q^\times h \in \sH \} = \sH_+ = \dom A = \dom T.
\]
Since $(Tf,g)=(Q^\times f,g) = (f, Qg)_{+,-}$, $f \in \dom T = \sH_+$ $g \in \sH$, it follows that
\[
  T^* = Q \uphar {\dom T^*} = (\bA + Z ) \uphar {\dom T^*}, \quad
  \dom T^*=\left\{h\in \sH: Q h\in\sH \right\},
\]
and thus, due to \eqref{form1},
\begin{align}\label{dom1a}
   \dom T^*
   &=\{h\in \sH: \left(I_{\sH}+(\bA {\pm} \I I_\sH )^{-1}Z\right) h\in \sH_+\},
   \end{align}
which proves \eqref{eq:singpert*}. Altogether, by \eqref{eq:singpert}, \eqref{dom1a},
\begin{align}
\nonumber
    \dom T\cap\dom T^*
    & = \{ h \in \sH : h \in \sH_+, \ (I_\sH+(\bA- \I I_\sH )^{-1}Z) h\in \sH_+ \}\\
\nonumber
    & = \{ h \in \sH : h \in \sH_+, \ (\bA- \I I_\sH )^{-1}Z h\in \sH_+ \} \\
\label{inter}
    & = \{ h \in \sH : h\in \sH_+, \ Zh\in \sH \}.
\end{align}
By \eqref{inter},  the assumption $\ran Z\cap \sH=\{0\}$ in \eqref{final} yields that
if $h\in\dom T\cap\dom T^*$, then $h\in \sH_+$ and $Zh=0$. Thus
\begin{equation}
\label{H_+Z}
   \dom T\cap \dom T^* = \sH_+\cap \ker Z;
\end{equation}
in particular, $\ker Z=\{0\}$ implies that $\dom T\cap\dom T^*=\{0\}$.

The condition $\dim\ran Z=n<\infty$ in \eqref{final} implies that $\dim\ran Z^\times=n<\infty$ and hence,
if we decompose the Hilbert space $\sH$ as
\[
   \sH= \ker Z\oplus (\ker Z)^\perp =\ker Z\oplus {\ran Z^\times},
\]
then \cite[Lemma 2.1]{MR0113146} shows that $\sH_+\cap\ker Z$ is dense in $\ker Z$.
Together with \eqref{H_+Z} this yields
$\codim(\dom T^*\cap\dom T) = \dim \ran Z^\times = n$.
If we view  $Z^\times$ as an operator in $\sH$ with domain $\dom Z^\times = \sH_+$ and denote its adjoint by $(Z^\times)^*$, then the assumption
$\ran Z\cap \sH=\{0\}$ in \eqref{final} implies that
\begin{align*}
   \dom (Z^\times)^* &= \big\{ g \in \sH: h\mapsto (Z^\times h,g) = (h,Zg)_{+,-} \mbox{ is continuous on } \dom \sH_+ \big\} \\
                     &= \big\{ g \in \sH: Zg \in \sH \big\} = \{0\}
\end{align*}
and hence $(\ker Z^\times)^\perp \!=\! \overline{\ran (Z^\times)^*} \!=\! \{0\}$ in $\sH$, i.e.\ $\ker Z^\times$ is dense in $\sH$.

Therefore, if we set $A_0:= A\uphar \ker Z^\times$, $\dom A_0:=\ker Z^\times$, then $A_0$ is a closed densely defined and symmetric operator in $\sH$; note that $T$ is an improper extension of $A_0$, i.e.\ $T$ is not a restriction of $A_0^*$.
Since $ \dom (Z^\times)^* \!=\! \{0\}$ is not dense in $\sH$, the operator $Z^\times$ viewed in $\sH$ is not closable and thus
$W(Z^\times)=\dC$ by \cite[Thm.\ V.3.4]{MR1335452}.

Now suppose that $A$ is not semibounded, i.e.\ $W(A)=\dR$,
and assume that $W(T)\!\ne\! \dC$.
Since $A_0 = A\uphar \ker Z^\times \subset T$, it follows that $\dR = W(A) \subset W(T)$ 
and hence $W(T)$ is contained in a half-plane of the form $P^{-}_{b_0}\!=\!\{z\!\in\!\dC: \IM z\!\le\! b_0\}$ or $P^{+}_{b_0}\!=\!\{z\!\in\!\dC:\IM z\!\ge\! b_0\}$ for some $b_0 \!\in\! \dR$.
Let, e.g.\  $W(T) \!\subset\! P^{-}_{b_0}$ and choose $b>b_0$. Since $W(Z^\times)=\dC$, there exists $h\!\in\!\dom T$, $\|h\|\!=\!1$,  such that $(Z^\times h,h)=\I b$.
Then $(Th,h)=(Ah,h)+ (Z^\times h,h)\notin W(T)$, a contradiction.

The last claim follows from the fact that $I_\sH + Z^\times (A-\lambda I_\sH )^{-1}$ is boundedly invertible for $\lambda \in \dC$, $\IM \lambda >1$, and that hence $(T-\lambda I_\sH )^{-1} = (A-\lambda I_\sH )^{-1} ( I_\sH + Z^\times (A-\lambda I_\sH )^{-1} )^{-1}$.
\end{proof}

\begin{remark}
Note that the operator  $T = Q^\times \uphar \{h\in \sH:Q^\times h\in \sH\}$
is said to arise from $A$ by a singular perturbation, see, e.g.\ \cite{MR1765471}.
\end{remark}

\begin{example}
\emph{
The motivation for the construction in this section came from a counter-example given in \cite[Ex.\ 3.5]{MR4083777}
to disprove the equivalence of different characterizations of the so-called essential numerical range for
unbounded operators. There $T=A+ \Phi(\cdot)g$ where $\Phi : \sH\to\dC$ is an unbounded linear functional which is $A$-bounded and $g\in \sH$ is a fixed element, see also \cite{MR2897729}. In this case, $\overline{\dom T\cap \dom T^*}=\{g\}^\perp$ and hence $ \codim\left(\dom T\cap \dom T^*\right)=1$.
A~con\-crete example in $\sH=L_2(\dR)$ is $Tf=\I f' + f(0) g$, $\dom (T)=W_2^1(\dR)$, with some function $g\in L_2(\dR)$.
}
\end{example}


\section{Maximal accretive and  maximal sectorial operators $T$ with $\dom T=\dom T^*$}
\label{sec:equality}

In this section we study the case $\dom T=\dom T^*$. By means of perturbation results we characterize classes of $m$-accretive and $m$-sectorial operators $T$ for which the domain equality $\dom T=\dom T^*$ holds.
At the end, we mention a more special situation of quasi-selfadjoint extensions $T$ of a closed, not necessarily densely defined symmetric operator $A$
for which $\dom T=\dom T^*$ was proved in \cite{MR2828331}, \cite{MR1383821}.

Note that the weaker property that $\dom T \cap \dom T^*$ is a core of $T$ is covered by Kato's result \cite[Thm.\ 5.1]{MR0151868}.
It shows that if $A$ is $m$-accretive, then $T\!:=\!A^{1/2}$ has the property that $\dom T \cap \dom T^*$ is a core of $T$; moreover, in this case $\dom A \!=\! \dom A^*$ implies that $\dom T\!=\!\dom T^*$, see \eqref{ravhalf}.

In the sequel we consider operators of the form $T=A+ \I B$ where both $A$, $B$ are unbounded selfadjoint and symmetric, respectively, operators. Note that, in general, even if $B$ is $A$-bounded, i.e.\ $\dom A \subset \dom B$ and hence $\dom A \subset \dom B^*$, only the inclusion $T^* \supset A - \I B^* = A - \I B$ holds.

The following proposition was proved in the note \cite[Prop.\ 1]{MR627724} for the case of positive definite $A$
with $\dom A=\dom B$, not using the perturbation arguments below.

\begin{proposition}
\label{rav1}
Let $A$ be a non-negative selfadjoint operator in $\sH$ and let $B$ be a symmetric operator with
$\dom A^{1/2} \subset \dom B$. Then the \vspace{-0.5mm}operator
\[
   T:=A + \I B,\quad \dom T:=\dom A,
\vspace{-0.5mm}
\]
has the following properties:
\begin{enumerate}
\item[{\rm i)}] $T$ is $m$-accretive and $T-aI$ is $m$-sectorial for every {$a\in \big(\!-\!\infty,\min \sigma(A)\big)$};
\item[{\rm ii)}] if $A$ is positive definite, then $T$ is $m$-sectorial with $0\in\rho(T)$;
\item[{\rm iii)}] $\dom T^*=\dom A$, $T^*=A-\I B$ \vspace{-0.5mm}and
\[
  \RE T:= \frac 12 (T+T^*)=A,\quad \IM T:= \frac 1{2\I} (T-T^*) =B{\uphar \dom A},
  	\vspace{-0.5mm}
\]
\item[{\rm iv)}] $\dom T^{1/2}=\dom T^{*1/2}=\dom A^{1/2}$.
\end{enumerate}
\end{proposition}

\begin{proof}
Clearly, $T$ is accretive. The assumption $\dom A^{1/2} \subset \dom B$ implies that $B$ is $A$-bounded with $A$-bound $0$,
see e.g.\ \cite[Prop.\ 2.1.19]{MR2463978}, and, since $B$ is symmetric, also that $B^*$ is $A^{1/2}$-bounded with $A$-bound $0$.
Hence, by \cite[Cor.\ 1]{MR0282244}, it follows that $T^*= A - \I B^* = A- \I B$, $\dom T^*=\dom A$, which proves the first claim in iii). Thus $T^*$ is accretive as well and so $T$ is $m$-accretive.

Let $a \in \big(\!-\!\infty, \min \sigma(A)\big) $ be arbitrary. Since $\dom (A-aI)^{1/2} = \dom A^{1/2}$, the operator $B(A-aI)^{-1/2}$
is bounded and we can estimate, for $f\in \dom T = \dom A$,
\begin{align*}
  \big| \IM  ((T-a I)f,f)\big| =
	\big|(Bf,f)\big| & = \big| (B(A-a I)^{-1/2} (A-a I)^{1/2} f,(A-a I)^{-1/2} (A-aI)^{1/2} f) \big| \\
	                 & \le \|B(A-aI)^{-1/2}\| \,\|(A-aI)^{-1/2}\| \,\|(A-aI)^{1/2}f\|^2 \\
			 			 & = \|B(A-aI)^{-1/2}\| \,\|(A-aI)^{-1/2}\| \,\RE \big( (T-aI) f,f \big).
\end{align*}
This proves that $T\!-\!aI$ is $m$-sectorial for all $a \in \big(\!-\!\infty, \min \sigma(A)\big)$; in particular, if $A$ is positive definite, we can choose $a=0$. This completes the proof of i) and ii).

Since $\dom T = \dom T^* = \dom A$, we have $\dom (T\pm T^*) = \dom A$ and so the second claim in iii) for the operators $\RE T$, $\IM T$ is  immediate.

Finally, if we fix $a \in \big(\!-\!\infty, \min \sigma(A)\big)$, \eqref{ravhalf} applied to the $m$-sectorial operator $T-aI$~yields
\[
  \dom (T- a I)^{1/2}=\dom (T^*-aI)^{1/2}  = \dom \mathfrak t_a = \dom \mathfrak a = \dom A^{1/2},
\]
where $\mathfrak t_a$ and $\mathfrak a$ are the forms induced by $T\!-\!aI$ and $A$, respectively. Now iv) follows if we use that
$\dom (T\!-\!aI)^{1/2}=\dom T^{1/2}$, $\dom (T^*\!-\!aI)^{1/2}=\dom T^{*1/2}$ by \cite[Lemma A2]{MR0138005}.
\end{proof}

\begin{theorem}
\label{rav2}
Let $\cB$ be a closed densely defined symmetric operator, let $F$ be a bounded positive definite selfadjoint operator in $\sH$,
and define a family $T(\mu)$, $\mu \in \dC$, of operators by
\begin{equation}\label{opfam}
   T(\mu):= \cB^*F\cB+\mu\cB,\quad \dom T(\mu):=\dom \cB^*F\cB, \qquad \mu \in \dC.
\end{equation}
Then $T(\mu)$, $\mu \!\in\! \dC$, is quasi-$m$-accretive, i.e.\ $T(\mu)\!+\! c I$ is $m$-accretive for some $c\!\in\!\dR$ with adjoint
\[
    T(\mu)^*\!=T(\overline{\mu})=\cB^*F\cB+\overline{\mu}\cB, \quad \dom T(\mu)^*\!=\dom T(\mu)=\dom \cB^*F\cB, \qquad \mu \in \dC,
\]
and it has the following properties:
\begin{enumerate}
\item[{\rm i)}] if $\mu = a \in \dR$, then $T(a)=\cB^*F\cB+a\cB$ is selfadjoint and bounded from below;
\item[{\rm ii)}] if $\mu \!=\! \I a$, $a\!\in\! \dR$, then $T(\I a)\!=\!\cB^*F\cB\!+\!\I a\cB$ is $m$-accretive,
$\dom T(\I a)^{1/2}\!=\!\dom T(\I a)^{*1/2}$ $=\!\dom\cB$ and, if $\cB$ has bounded inverse, $T(\I a)$ is $m$-sectorial with $0\in\rho(T(\I a))$;
conversely if $\ker \cB=\{0\}$ and if $\,T(\I a)$ is sectorial, then $\cB$ has bounded inverse.
\end{enumerate}
\noindent
If, in addition, $\,\cB$ is non-negative, we also have:
\begin{enumerate}
\item[{\rm iii)}]
if $\RE \mu \!\ge\! 0$, then $T(\mu)=\cB^*F\cB+\mu \cB$ is $m$-accretive, $\dom T(\mu)^{1/2}\!=\!\dom T(\mu)^{*1/2}\!=\!\dom\cB$
and, if $\cB$ is positive definite, $T(\mu)$ is $m$-sectorial with $0\in\rho(T(\mu))$;
\item[{\rm iv)}]
if $\RE \mu \!\ge\! 0$ and $\IM \mu\! \ge\! 0$, i.e.\ $\mu \!=\! a \e^{\I \alpha}$, $a \!\in\! \dR$, $\alpha \!\in\! [0,\frac \pi 2]$,
then, in addition to {\rm iii)}, $T(a\e^{\I \alpha})\!=\!\cB^*F\cB\!+\!a\e^{\I \alpha}\cB$
is $m$-dissi\-pative with
\begin{equation}
\label{numrT}
   W(T(a\e^{\I \alpha}))\subset \{ z\in \dC: 0\le \arg z \le \alpha\},
\end{equation}
$T(a\e^{\I \alpha})$ is $m$-sectorial with semi-angle $\alpha$ if $\alpha\! <\! \frac \pi 2$ and $T(a)$ is non-negative if $\alpha=0$.
\end{enumerate}
\end{theorem}

\begin{proof}
Since $F$ is a positive definite bounded selfadjoint operator, the operator $F^{1/2}\cB$ has dense domain $\dom(F^{1/2}\cB)=\dom \cB$ and is closed  with adjoint
\[
   (F^{1/2}\cB)^*=\cB^*F^{1/2},\quad \dom(F^{1/2}\cB)^*=F^{-1/2}\dom\cB^*.
\]
Thus, by \cite[Thm.\ V.3.24]{MR1335452} also due to von Neumann,
$\cB^*F\cB \!=\! (F^{1/2}\cB)^* F^{1/2}\cB$ is a non-ne\-ga\-tive selfadjoint operator, $\dom \cB^*F\cB$ is a core of $\cB$
and $\dom (\cB^*F\cB)^{1/2} \!=\! \dom F^{1/2}\cB \!=\! \dom \cB$.

The latter implies that $\cB$ is  $\cB^*F\cB$-bounded with $\cB^*F\cB$-bound 0,
see e.g.\ \cite[Prop.\ 2.1.19]{MR2463978}, and hence so is $\mu \cB$ for any $\mu\in \dC$. This yields, first, that $T(\mu)^*$ has the claimed form
by \cite[Cor.\ 1]{MR0282244} and, secondly, that $A(\mu):= \cB^*F\cB + \RE \mu \,\cB$ is selfadjoint and bounded from below by \cite[Thm.\ V.4.3, Thm.\ V.4.11]{MR1335452}, i.e.\ $A(\mu) + c I \ge 0$ for some $c\in \dR$.  Now Theorem \ref{rav1} applied to the operators $A(\mu)+c I$ and $\IM \mu \,\cB$ shows that $T(\mu)+c I$ is $m$-accretive.

i) The claims follow from what was shown above since in this case $\IM \mu \!=\! 0$  and~$T(a)\!=\!A(a)$.

ii) By what we proved above, the operators $\cB^*F\cB$ and $\cB$ satisfy the assumptions of Pro\-position \ref{rav1}.
Since $F$ is positive definite and $\cB$ has bounded inverse, $\cB^*F\cB$ is positive definite.
Thus all claims in ii) except for the last one follow from Proposition \ref{rav1}. To prove the last claim, let
$\ker \cB=\{0\}$ and let $T(\I a)$ be sectorial.
The latter implies that, for some~$\gamma \!>\!0$,
\[
   \left|\left(\cB f,g\right)\right|\le \gamma \left\|\cB f\right\|\left\|\cB g\right\|, \quad f,g\in\dom\cB,
\vspace{-0.5mm}
\]
and \vspace{-0.5mm} hence
\[
   \left|\left(\psi,\cB^{-1}\phi\right)\right|\le \gamma \left\|\psi\right\|\left\|\phi\right\|, \quad  \psi,\phi\in\dom\cB^{-1}.
\]
It follows that $\left\|\cB^{-1}\phi\right\|\le\left\|\phi\right\|$ $\phi\in\dom\cB^{-1}$.

iii)
Since $\RE\mu \ge 0$ and $\cB$ is non-negative, $A(\mu)\!=\!\cB^* F \cB \!+\! \RE \mu \,\cB$ is non-negative. Now all
claims follow from Proposition \ref{rav1} applied to the operators $A(\mu)$ and~$\IM \mu \,\cB$.

iv) The enclosure \eqref{numrT} is immediate since $\cB^*F\cB$ is always non-negative and $\cB$ is assumed to be non-negative.
Together with $\RE \mu = a \cos \alpha \ge 0$, the latter implies that $A(\mu) \ge 0$, so that we can choose $c=0$ and
$T(a \e^{\I \alpha})$ is $m$-accretive and, due to \eqref{numrT}, $m$-dissipative.
\end{proof}

\begin{corollary}
\label{rav2cor}
Under the assumptions of Theorem {\rm \ref{rav2}} with $\,\cB$ non-negative, the operators
\begin{alignat*}{2}
  T_1&\!:=\! T(\I)\!=\!\cB^*F\cB\!+\!\I \cB, & T_1^*&= T(-\I)\!=\!\cB^*F\cB\!-\!\I \cB,\\
  T_2&\!:=\! \I T(\I)^* \!=\! \cB+\I\cB^*F\cB, & T_2^*& = -\I T(\I) = \cB-\I\cB^*F\cB,   \\
  T_3&\!:=\!T(1\!+\!\I)= \cB^*F\cB\!+\!\cB + \I \cB,\ \ &  T_3^*&=\! T(1\!-\!\I) \!=\! \cB^*F\cB\!+\!\cB - \I \cB, \\
  T_4&\!:=\! (1\!+\!\I) T(\I)^* = \cB^*F\cB\!+\!\cB + \I (\cB^*F\cB\!-\!\cB),\ \ &  T_4^*&=\! (1\!-\!\I) T(\I) \!=\! \cB^*F\cB\!+\!\cB - \I (\cB^*F\cB\!-\!\cB),
\end{alignat*}
all defined on  $\dom T_i = \dom T_i^*= \dom \cB^*F\cB$, $i\!=\!1,2,3,4$, satisfy
\begin{itemize}[leftmargin=6mm]
\item[{\rm i)}] $T_1$, $T_2$ are $m$-accretive, $m$-dissipative and, if $\cB$ is positive definite, $0\!\in\! \rho(T_1)$, $0\!\in\! \rho(T_2)$;
\item[{\rm ii)}] $T_3$, $T_4$ are $m$-sectorial with semi-angle $\frac \pi 4$ and, if $\cB$ is positive~definite,
$0\!\in\! \rho(T_3)$,~$0\!\in\! \rho(T_4)$,
\end{itemize}
and $\dom T_i^{1/2}\! =\! \dom T_i^{*1/2} \!=\! \dom \cB$, $i\!=\!1,2,3$.
Besides, if $\,T_2$ is sectorial, then $\cB$ is bounded.
\end{corollary}

\begin{proof}
All claims except for the last one follow from the properties of the operators $T(\I)$ and $T(1\!+\!\I)$ in Theorem~\ref{rav2} iii) and iv) and from the definitions of $T_i$, $i=1, 2,3,4$. To prove the last claim, let $T_2$ be sectorial. Then, for some $\gamma >0$,
\[
   \left(F\cB f,\cB f\right)\le\gamma(\cB f,f),\quad f\in\dom\cB^*F\cB,
\]
and hence, since $F$ is positive definite, for some $\beta>0$,
\[
   \|\cB f\|^2\le \beta\left(F\cB f,\cB f\right)\le\beta\gamma(Bf,f)\le\beta\gamma \|\cB f\| \,\|f\|,\quad f\in\dom\cB^*F\cB,
\]
which implies that $\cB$ is bounded.
\end{proof}

\begin{open}
\label{realpart}
If $\cA$ is an $m$-sectorial operator such that $\dom \cA=\dom \cA^*$, is it true that $\RE \cA= \frac 12 (\cA+\cA^*)$ is a selfadjoint operator?
\end{open}

In general, it is only clear that the Friedrichs extension of $\RE \cA$ of an $m$-sectorial operator $\cA$ coincides
with the selfadjoint operator $\cA_R$ called real part of $\cA$ in \cite[Sect.\ VI.3.1]{MR1335452}; the latter is defined as the non-negative selfadjoint operator associated with the real part ${\rm Re}\, {\mathfrak a} = \frac 12({\mathfrak a} + {\mathfrak a}^*)$ of the form ${\mathfrak a}$ corresponding to $\cA$.

The operators $T_i$, $i=1,2,3,4$, in Corollary \ref{rav2cor} show that for merely $m$-accretive $\cA$, the answer is negative. In fact,
the real part of the $m$-accretive operator $T_1 = \cB^*F\cB\!+\!\I \cB$ 
is non-negative selfadjoint, {whereas}
the real part of the $m$-accretive operator  $T_2=\cB+\I\cB^*F\cB$ 
is the non-closed non-negative densely defined symmetric operator $\cB\uphar\dom\cB^*F\cB$,
and if $\cB$ is unbounded non-negative with $\ker \cB=\{0\}$, but the inverse of $\cB$ is not bounded, then $T_1$, $T_2$ are not sectorial by Theorem \ref{rav2} and Corollary \ref{rav2cor}.
On the other hand,
the $m$-sectorial operators $T_3$, $T_4$ do have
non-negative selfadjoint real parts $\cB^*F\cB\!+\!\cB$.

\smallskip

The following proposition allows us to construct $m$-accretive and $m$-sectorial operators~$T$ with $\dom T\!=\!\dom T^*$ as in Theorem \ref{rav2} and Corollary \ref{rav2cor} which even have compact resolvent.

\begin{proposition}
\label{compres1}
Let $\cB$ be a closed densely defined symmetric operator in $\sH$ with bounded inverse
admitting a selfadjoint extension $\wt \cB$ in $\sH$ with compact inverse $\wt\cB^{-1} {\in \bB(\cH)}$.
Then,  for every bounded positive definite selfadjoint operator $F$ in $\sH$, the operator $(\cB^*F\cB)^{-1}$
{defined on $\dom (\cB^*F\cB)^{-1} \!\!=\! \ran \cB^*F\cB$} is compact.
\end{proposition}

\begin{proof}
The operators $C:=\cB^*F\cB$ and $\wt C:=\wt \cB F\wt \cB$ are both selfadjoint and positive definite by \cite[Thm.\ V.3.24]{MR1335452}.
Then, by the assumptions on $\wt \cB$ and $F$, the inverse
\[
  \wt C^{-1}=\wt \cB^{-1}F^{-1}\wt \cB^{-1} {\in \bB(\cH)}
\]
is compact and, by \cite[Chapt.\ VI, (2.22)]{MR1335452},
\begin{align*}
   \|\wt C^{1/2}\wt f\|^2&= \|F^{1/2}\wt \cB\wt f\|^2,\quad \wt f\in \dom \wt C^{1/2}=\dom\wt \cB,\\
   \| C^{1/2}f\|^2 &= \|F^{1/2}\cB f\|^2,\quad f\in\dom C^{1/2}=\dom\cB,
\vspace{-2mm}
\end{align*}
which implies $\|\wt C^{1/2} f\|\!=\!\|C^{1/2}f\|$, $f\!\in\! \dom C^{1/2}$.
Since $\dom \cB\!\subset\!\dom\wt \cB$, we have $\ran C^{-1/2} = $ ${\dom C^{1/2} \!\subset\! \dom \wt C^{1/2}}$ and hence
$
\| \wt C^{1/2} C^{-1/2} h \|\!=\!\|h\|$, $h\!\in\! {\dom C^{1/2}}
$.
It follows that $V\!:=\!\wt C^{1/2} C^{-1/2}$ is an isometry on $\dom V \!:=\! \dom C^{-1/2}$.
Since $C^{-1/2}\!=\!\wt C^{-1/2}V$ and $\wt C^{-1/2}$ is compact,
$C^{-1/2}$ is compact and hence so is $C^{-1}\!=\!\left(\cB^*F\cB\right)^{-1}$ by \cite[Thm.~V.3.49]{MR1335452}.
\end{proof}

\begin{corollary}\label{compres11}
Under the assumptions of Proposition {\rm \ref{compres1}}, the operator {$T_1=\cB^*F\cB+\I\cB$, $\dom T_1\!=\!\dom\cB^*F\cB$, in Corollary {\rm \ref{rav2cor}}}
is $m$-sectorial and has compact inverse. If, in addition, $\cB$ is positive definite, then the same holds for all operators $T_i$, $i\!=\!2,3,4$, in Corollary~{\rm\ref{rav2cor}}.
\end{corollary}

\begin{proof}
By Theorem \ref{rav2} ii) the operator $T_1$ is $m$-sectorial and $0\in \rho(T_1)$. \vspace{-1mm} Since
\[
   T_1^{-1}
	=\left(\cB^*F\cB^*\right)^{-1/2}\left(I+\I\left(\cB^* F\cB\right)^{-1/2}\cB\left(\cB^*F\cB\right)^{-1/2}\right)^{-1}\left(\cB^*F\cB^*\right)^{-1/2}
\vspace{-1mm}
\]
and  $\left(\cB^*F\cB^*\right)^{-1/2}$ is compact,  $T_1^{-1}$ is compact as well. The proof for $T_i$, $i=2,3,4$, is similar if we use Corollary \ref{rav2cor}.
\end{proof}

In the following we present some concrete examples of differential operators to illustrate Theorem \ref{rav2} and Corollary \ref{rav2cor}.

\begin{example}
\emph{
For a subinterval $I\subset \dR$, we denote by $W^k_2(I)$, $k\in \dN$, the Sobolev spaces of order $k$
corresponding to $L_2(I)$, and we abbreviate the first order derivative by $\D:= \frac{\rm d}{{\rm d} x}$. }

\smallskip

\emph{ 1) In $\sH=L_2(0,1)$ we consider the operators
\[
   \cB:=-\I\,\D, \quad \dom \cB=\{f\in W^1_2(0,1):f(0)=f(1)=0\},
\]
and, for an arbitrary function $p \in L_\infty(0,1)$ with ess\,inf $p>0$, the multiplication operator
\[
(F_{p}h)(x)=p(x)h(x),\quad  h\in L_2(0,1), \ x \in (0,1).
\]
Then the operator $T_1=\cB^*F_{p}\cB+\I\cB$, $\dom T_1\!=\!\dom\cB^*F_p\cB$, is given by
\[
   T_1=-\D p\D +\D ,\quad \left\{f\in W^1_2(0,1) : f(0)=f(1)=0,\, (pf')'\in L_2(0,1)\right\}.
\]
Since $\cB$ has bounded inverse, but is not non-negative, Theorem \ref{rav2} ii) and Corollary \ref{rav2cor} i) show that
$T_1$ is $m$-sectorial with adjoint
\[
   T_1^*=-\D p\D -\D ,\quad \dom T_1=\dom T_1^*,
\vspace{-1mm}
\]
and
\[
   \dom T_1^{1/2}=\dom T_1^{*1/2} = \{f\in W^1_2(0,1):f(0)=f(1)=0\}.
\]
Clearly, $T_1$ has compact resolvent.}

\smallskip

\emph{
2) In $\sH=L_2(0,\infty)$ we consider the operators
\[
   \cB:=-\D^2,\quad \dom \cB= \{f\in W^2_2(0,\infty): f(0)=f'(0)=0\},
\]
and, for an arbitrary function $p \in L_\infty(0,\infty)$ with ess\,inf $p>0$, the multiplication operator
\[
    (F_{p}h)(x)=p(x)h(x),\quad  h\in L_2(0,\infty), \ x\in (0,\infty).
\]
Then the operator $T_1=\cB^*F_{p}\cB+\I\cB$, $\dom T_1\!=\!\dom\cB^*F_p\cB$, is given by
\[
   T_1=\D^2 p \D^2 -\I\D^2, \quad  \dom T_1 =\left\{ f\in W^2_2(0,\infty):f(0)=f'(0)=0, \,\D^2 p \D^2 \in L_2(0,\infty)\right\}.
\]
Since $\cB$ is non-negative with $\sigma(\cB)=[0,\infty)$, Theorem \ref{rav2} ii) and Corollary \ref{rav2cor} i) show that
$T_1$ is $m$-accretive and $m$-dissipative with adjoint
\[
   T_1^*=\D^2 p \D^2 +\I\D^2, \quad \dom T_1^*=\dom T_1,
\]
and
\[
    \dom T_1^{1/2}=\dom T_1^{*1/2}=\{f\in W^2_2(0,\infty):f(0)=f'(0)=0\}.
\]
Moreover, by Proposition \ref{rav1}, $T_1-aI$ is $m$-sectorial for every $a\in (-\infty,0)$.
}

\end{example}

\begin{remark}
In \cite{MR1155715} an example of a hyponormal unbounded $m$-sectorial operator $T$ with $\dom T=\dom T^*$ was constructed, which is given by a tri-diagonal semi-infinite matrix in $l_2(\dN)$; note that, by definition, a hyponormal operator $T$ satisfies $\dom T\cap \dom T^* = \dom T$.
\end{remark}

We conclude this section with some known results on the coincidence of domains of an operator and its adjoint when both are extensions of
a closed non-densely defined symmetric operator $\dot A$ in a Hilbert space $\sH$.

A closed densely defined operator $T$  in $\sH$ is called  \emph{quasi-selfadjoint extension} of $\dot A$  if $T\supset \dot A$ and $ T^*\supset \dot A$. If $\dot A$ is densely defined, then $T$ is a quasi-selfadjoint extension of $\dot A$ if and only if $T\subset \dot A^*$; in this case the equality $\dom T\!=\!\dom T^*$ yields that $T$ is a selfadjoint extension of $\dot A$.

Vice versa, if $T$ is a closed densely defined non-selfadjoint operator in $\sH$ and the symmetric operator $\dot A$ given \vspace{-1mm}by
\begin{equation}
\label{oppa}
   \dom \dot A:=\{f\in\dom T\cap\dom T^*: Tf=T^*f\}, \quad {\dot A := T \uphar \dom \dot A},
\vspace{-1mm}
\end{equation}
satisfies $\dom \dot A\!\ne\!\{0\}$, then $\dot A$ is a closed symmetric operator \cite{MR2828331}, 
\cite{MR1383821} and $T$ is a quasi-selfadjoint extension of $\dot A.$ Besides, if $\dom\dot A$ is dense, then $\dom \dot A$ is neither a core of $T$ nor of~$T^*$.

In \cite[Chapt.\ 1, Thm.\ 2.6]{MR1383821} certain properties of the operator $\dot A$ defined by \eqref{oppa} were established
in the case $\rho(T)\ne\emptyset$ and $\dom T=\dom T^*$. Moreover,
in \cite[Chapt.\ 3, Thm.\ 2.1]{MR1383821} criteria for the equality $\dom T=\dom T^*$ were derived in terms of characteristic functions
for quasi-selfadjoint extensions $T$ of a non-densely defined symmetric operator with finite deficiency indices.
Further, in \cite{MR2828331} special classes of closed non-densely defined symmetric operators, called \emph{regular $O$-operators}, were
studied, and the domain equality $\dom T=\dom T^*$ was proved for a class of quasi-selfadjoint extensions $T$ of a regular $O$-operator in
\cite[Thm.~4.1.12]{MR2828331}.
Finally, in \cite[Thm.\ 2]{MR715553}, using a factorization of the perturbation determinant, a criterion for the
equality $\dom T=\dom T^*$ for a class of $m$-dissipative operators with nuclear imaginary part was established.


\section{Stability results for $\dom \cA \cap \dom \cA^*$: Accretive and non-accretive operators}
\label{sec:additive-perturb}

In this section we provide stability results for the domain intersections $\dom \cA \cap \dom \cA^*$ of an $m$-sectorial operator $\cA$ in a Hilbert space $\sH$.
More precisely, we construct operators $\wh T$  in a possibly larger Hilbert space $\wh \sH$ such that $\dom \wh T \cap \dom \wh T^*$ has the same properties as $\dom \cA \cap \dom \cA^*$,
$\rho(\wh T)\ne \emptyset$ and either $\wh T$ is $m$-accretive but not sectorial or $\wh T$ is not even accretive, i.e.\ $T$ has full numerical range $W(\wh T) = \dC$.

\begin{proposition}
\label{addapert}
Let $A$ be an unbounded $m$-sectorial operator and $B$ a bounded selfadjoint operator
in the Hilbert \vspace{1mm} space~$\sH$. Then $T\!:=\!A\!+\!\I B$
satisfies $\dom T \cap \dom T^* = \dom A \cap \dom A^*$ and  has the following properties.
\begin{enumerate}
\item[{\rm i)}]
If $\,\ran A\!=\!H$ $($and hence $\ker A = \{0\})$,
then $T$ is $m$-\vspace{1mm} sectorial in~$\sH$;
\item[{\rm ii)}]
If $\,\ker A\!=\!\{0\}$ \vspace{-2mm} and
\begin{equation}
\label{ceghtv}
 \sup\limits_{u\in\dom A {\setminus \{0\}}} \cfrac{{|}(Bu,u){|}}{\RE (Au,u)}=\infty
\end{equation}
$(${and hence} $\ran A\!\ne\! \sH)$, then
$T$ is $m$-accretive, but not sectorial, i.e.\ there is no sector $\overline{\cS(\alpha)}=\left\{z\in\dC:|\arg z |\le \alpha \right\}$
with $\alpha\in(0,\pi/2)$ that contains $W(T)$, and the same holds for $T^*\!$.
Moreover, the $m$-accretive operators in $\sH\oplus \sH$ given~by
\begin{alignat}{2}
\label{chaitat1}
   \wh T&:=\begin{bmatrix}A\!+\!\I B\!&0\cr 0& \!A^*\!-\!\I B \end{bmatrix},
   	& \quad \dom \wh T&=\dom A\oplus\dom A^*\!\!,\\
\label{chaitat2}
   \wt T&:=\begin{bmatrix}A&\I B\cr \I B&A\end{bmatrix},
	 	& \quad \dom\wt T&=\dom A\oplus\dom A,
\end{alignat}
satisfy
$\dom \wh T \cap \dom \wh T^* \!=\! (\dom A \cap \dom A^*) \!\oplus\!  (\dom A \cap \dom A^*) \!=\! \dom \wt T \cap \dom \wt T^*\!$,
and neither $W(\wh T)$ nor $W(\wt T)$ are contained in any sector $\cS$ of opening angle $<\pi$ in the closed right half-plane.
\end{enumerate}
\end{proposition}

\begin{proof}
Clearly, since $B$ is bounded, $T\!+\!aI$ is $m$-sectorial for every $a>0$. Therefore $\overline{W(T\!+\!aI)}=\overline{W(T^*\!+\!aI)}^* = \{ z\in \dC: \bz \in \overline{W(T^*\!+\!aI)}\}$
since both $W(T\!+\!aI)$ and $W(T^*\!+\!aI)$ are dense in the numerical ranges of their respective forms $W(\mathfrak t_a)$ and $W(\mathfrak t_a^*)$, respectively, and $W(\mathfrak t_a)=W(\mathfrak t_a^*)^*$
by \cite[Cor.~VI.2.3 and Thm.~VI.2.5]{MR1335452}. This shows that $\overline{W(T)}=\overline{W(T^*)}^*$.

By \cite[Thm.\ VI.3.2]{MR1335452},  there exist an unbounded selfadjoint operator $L$ and a bounded selfadjoint operator $G$ in $\sH$
such that $A=L(I+\I G)L$.
This implies that $\dom A \subset \dom L$, $\ker A= \ker L$, $\ran A \subset \ran L$ and thus, for $T:=A+\I B$, $\dom T=\dom A$,
\begin{equation}
\label{sectAiB}
   \RE(Tu,u) \!=\! \|Lu\|^2\ge 0, \quad
   \IM (Tu,u) \!=\!  (GLu,Lu) \!+\! (Bu, u),  \quad u \!\in\! \dom T,
\end{equation}
hence, in particular, $T$ is accretive.
Since $A$ is $m$-sectorial, $A^*$ is sectorial by \cite[Thm.~VI.2.5]{MR1335452} and thus also $T^*=A^*-\I B$, $\dom T^*=\dom A^*$, is accretive.
This proves that both $T$ and $T^*$ are $m$-accretive.

i) Suppose that $\ran A\!=\!\sH$. Then, since $A$ is closed, we have $\ker A\!=\!\ker A^*\!=\!(\ran A)^\perp \!=\!\{0\}$. \linebreak Further,
$\ker L \!=\! \{0\}$, $\ran L\!=\!\sH$ and so there exists $m>0$ with $\|Lu\|\!\ge\! m\|u\|$, $u\!\in\!\dom L$.~Thus
\[
   |(Bu,u)|\le \|B\|\, \|u\|^2\le{\|B\| m^{-1}} \|Lu\|^2, \quad u\in\dom L,
\]
and so for $T=A+\I B,$ $\dom T=\dom A$, \eqref{sectAiB} yields that
\[
  |\IM (Tu,u) | \le \big( \|G\| + \|B\| m^{-1}  \big) \,\RE(Tu,u), \quad u \in \dom T,
\]
and hence $T$ is $m$-sectorial.

ii)  Suppose that $\ker A = \{0\}$ and \eqref{ceghtv} holds.
Then $\overline{\ran A} =  (\ker A^*)^\perp = (\ker A)^\perp = \sH$. Further, since $B$ is bounded,
\eqref{ceghtv} necessitates that there exists a sequence  $\{u_n\}_{n\in\dN} \subset \dom A$, $\|u_n\|=1$, $\RE(A u_n,u_n)\to 0$, $n\to \infty$, and thus, since $A$
is sectorial, $(A u_n,u_n)\to 0$, $n\to \infty$. This shows that $0 \in \sigma_{\rm app}(A) \subset \sigma(A)$ and hence
$\overline{\ran A}\ne \sH$.
By \eqref{sectAiB} and noting $\ker L = \ker A = \{0\}$, we have
\[
 \cfrac{\IM (Tu,u)}{\RE(Tu,u)} = \cfrac{(GLu,Lu)}{\|Lu\|^2}+\cfrac{(Bu,u)}{\|Lu\|^2}
 =: \mu_u + t_u, \quad u \in \dom T \setminus \{0\},
\]
where $\mu_u \in [-\|G\|, \|G\|]$ since $G$ is bounded. Now \eqref{ceghtv} implies that $t_{u_n} \to \infty$ or $t_{u_n} \to -\infty$ for some sequence
$\{u_n\}_{n\in\dN} \subset \dom A \setminus\{0\}$ and thus
\[
   \sup\limits_{u\in\dom T    {\setminus \{0\}}}\cfrac{|\IM (Tu,u)|}{\RE(Tu,u)}=\infty,
\]
which proves that $T$ is not sectorial. Since $\overline{W(T)}=\overline{W(T^*)}^*$, the same is true for $T^*$.

Clearly, $\wh T$ and $\wt T$ are of the form $\wh T \!=\! \wh A \!+\!\I \wh B$, $\wt T \!=\! \wt A \!+\!\I \wt B$ with $\wh A$, $\wt A$ $m$-sectorial
and bounded selfadjoint $\wh B$, $\wt B$, and thus $\wh T$ and $\wt T$ are $m$-accretive by what was shown~above.

By ii), for every $\epsilon \in (0,\pi/2]$, we either have
$\overline{W(T)} \cap  \Sigma(\epsilon) \ne \emptyset$ or $\overline{W(T)} \cap  \Sigma(\epsilon)^* \ne \emptyset$
for the sector $\Sigma(\epsilon):=\{ z\!\in\! \dC: \pi/2\!-\!\epsilon \!<\! \arg z \!\le\! \pi/2 \}$.
Since $\overline{W(T)}\!=\!\overline{W(T^*)}^*\!$, e.g.\ in the former case, it follows that
$\overline{W(T^*)} \cap  \Sigma(\epsilon)^* \ne \emptyset$. Hence, for $\overline{W(\wh T)} = \overline{W(T)} \cup \overline{W(T^*)}$,
for every $\epsilon \in (0,\pi/2]$ we have $\overline{W(\wh T)} \cap  \Sigma(\epsilon) \ne \emptyset$ and $\overline{W(\wh T)} \cap  \Sigma(\epsilon)^* \ne \emptyset$, which proves the claim for $\wh T$.

The claim for $\wt T$ follows if we let $u_\pm = (u,\pm u)^{\rm t} \in \dom A \oplus \dom A$ with $u\in\dom A$,
$\|u\|= 1/\sqrt2$ so that $\|u_\pm\| = 1$, \vspace{-1mm} observe
\[
  \big(\wt T u_\pm,u_\pm \big) = \frac {(Au,u)}{\|u\|^2} \pm \I  \frac {(Bu,u)}{\|u\|^2} \in W(\wt T),
  \vspace{-1mm}
\]
and take into account \eqref{ceghtv} arguing as above.
\end{proof}

\begin{remark}
The operators $T$, $\wh T$ and $\wt T$ in Proposition \ref{addapert} ii) become sectorial
after a right shift, i.e.\ $T$, $\wh T\!+\!aI$ and $\wt T\!+\!a I$ with $a>0$ are boundedly invertible and $m$-sectorial.
\end{remark}

\begin{remark}
If $B$ is non-negative, then condition \eqref{ceghtv} in Proposition \ref{addapert} ii) is equivalent~to
\[
   \ran B^{1/2}\nsubseteq\ran A_R^{1/2}
\]
where $A_R\!:=\!L$ is the non-negative selfadjoint operator associated with the real part $\RE {\mathfrak a}$ of the closed sectorial form ${\mathfrak a}$ associated with $A$. Note that
$\ker A_R \!=\! \ker A \!=\! \{0\}$ was assumed in ii) and that \eqref{ceghtv} equally holds with $\dom A$ replaced by $\dom A_R^{1/2}$
since $\dom A$ is a core of ${\mathfrak a}$, and thus for $\RE {\mathfrak a}$ and~$A_R^{1/2}$.

Indeed, {\eqref{ceghtv} does not hold} if and only if the densely defined operator
$D:= B^{1/2}A_R^{-1/2}$\!, $\dom D=\ran A_R^{1/2}$, is bounded.
The latter is, in turn, equivalent to $D^*$ being bounded and everywhere defined by \cite[Satz~2.40 c)]{MR1887367}.
Since $B$ is bounded and both $B$, $A_R^{1/2}$ are selfadjoint, we have $D^* = A_R^{-1/2} B^{1/2}$ by \cite[Satz~2.43 b)]{MR1887367}
and so $D^*$ is everywhere defined if and only if $\ran B^{1/2}\subseteq\ran A_R^{1/2}$.
\end{remark}

Finally, we show that all possible phenomena for $\dom T \cap \dom T^*$ may also occur for closed operators
whose numerical range is the entire complex plane, $W(T)=\dC$, with non-empty resolvent set, or even with compact resolvent.

\begin{proposition}
\label{prprt1}
Let $T$ be a closed densely defined linear operator in $\sH$
that is not of the form $T=\eta(A+\I C)$ with a symmetric operator $A$, a bounded selfadjoint operator $C$ and $\eta \in \dC$.
Then, for all bounded operators $X$, $Y\in\bB(\sH)$, the closed linear operator $\wh T$ given by 
\[
    \wh T=\begin{bmatrix} T&X\cr Y&-T \end{bmatrix}, \quad
   \dom\wh T=\dom T\oplus\dom T,
\]
in $\sH \oplus \sH$ satisfies
\[
   \dom \wh T \cap \dom \wh T^* = (\dom T \cap \dom T^*) \oplus (\dom T \cap \dom T^*), \quad
	 W(\wh T)=\dC.
\]
Moreover, if $\,T$ has compact resolvent, then so does $\wh T$ and, for $X=0$ or $Y=0$, i.e.\ for
\[
   \wh T=\begin{bmatrix} T&X\cr 0&-T \end{bmatrix}  \quad \mbox{or } \quad
	 \wh T=\begin{bmatrix} T&0\cr Y&-T \end{bmatrix},
\]	
$\,0\!\in\!\rho(T)$ implies $0 \!\in\! \rho\big(\wh T\big)$.
\end{proposition}

\begin{proof}
Suppose that $W(\wh T)\ne \dC$. Since $W(\wh T)$ is convex, there exist $\nu$, $\mu \!\in\! \dC$ such that the numerical range of
$\nu \wh T\!+\!\mu I$ is contained in the closed upper half-plane,~i.e.,
\[
   \nu W(\wh T)+\mu\subseteq \{ z \in \dC : \IM z \ge 0\}.
\]
Since $W(T) \cup W(-T) \subseteq W(\wh T)$, it follows that, for every $z\in W(T)$,
\[
   \IM(\nu z+\mu)\ge 0,\quad
   \IM(-\nu z+\mu)\ge 0
\]
and hence
\[
   -\IM \mu\le \IM(\nu z)\le\IM\mu.
\]
This means that $W(\nu T)$, and hence $W(T)$, is contained in a strip. By \cite[Lemma~5.1]{MR3940390}, the latter implies that $T$ is of the form excluded by assumption, a contradiction.

The remaining claims are immediate since $X$ and $Y$ are bounded.
\end{proof}

\begin{remark}
\label{corso}
If $T$ were of the form $T=\eta(A+\I C)$ with symmetric $A$, bounded selfadjoint $C$ and $\eta \in \dC$, then we would have
$\dom T \subseteq \dom T^*$ and hence, in particular, in this case $\dom T \cap \dom T^* = \dom T$ is always dense and a core of $T$.
\end{remark}

By Remark \ref{corso}, in all possible cases (1) to (5) for $\dom T \cap \dom T^*$, see Introduction, the operators $T$ satisfy the assumptions of Proposition \ref{prprt1}.
Thus only the cases  (6), (7) remain to be considered where $\dom T \cap \dom T^*$ is a core or even $\dom T = \dom T^*$, see Section \ref{sec:equality}.

\begin{proposition}\label{lauss}
Let $\cB$ be a positive definite closed unbounded symmetric
operator in $\sH$, $F$ a positive definite bounded selfadjoint operator, $T_1=\cB^*F\cB+\I\cB$, $\dom T_1\!=\!\dom\cB^*F\cB$,
as in Theorem {\rm \ref{rav2}}, Corollary {\rm \ref{rav2cor}}, and $X\in \bB(\sH)$ bounded.
Then the operator
\[
   \wh T=\begin{bmatrix}T_1&0\cr X&-T_1  \end{bmatrix}, \quad {\dom \wh T = \dom T_1 \oplus \dom T_1,}
\]
in the Hilbert space $\sH\oplus \sH$ satisfies
\[
   \dom \wh T=\dom \wh T^*,\quad W(\wh T)=\dC, \quad 0\in\rho(\wh T),
\]
and if $\,T_1$ has compact resolvent, so does $\wh T$.
\end{proposition}

\begin{proof}
The claims follow from Proposition \ref{prprt1} if we
show that $T_1$ is not of the form $\eta(A+\I C)$ where $A$ is symmetric and $C$ is bounded selfadjoint. 
Otherwise, for some $x$, $y\in\dR$,
\[
   T_1  =(x+\I y)(A+\I C)
	\quad
	T_1^*	=(x-\I y)(A^*-\I C).
\]
Because $\dom T_1=\dom T_1^*$ by Theorem \ref{rav2} ii) and $C$ is bounded, it follows that $\dom A = \dom T_1=\dom T_1^* = \dom A^*$ and hence
$A$ is selfadjoint. Taking scalar products with all $f\in \dom T_1 = \dom \cB^*F\cB$ on both sides of
\[
   T_1=\cB^*F\cB+\I\cB=(xA-yC)+\I(xC+yA)
\]
and then imaginary parts, we conclude that the operator $S\!:=\!\cB \uphar \dom \cB^*F\cB - (xC+yA)$ satisfies $W(S)\!=\!\{0\}$. This implies that
$S$ is bounded and further that $S=0$, see \cite[Thm.\ 2.51]{MR1887367}, i.e.\
$\cB{\uphar \dom \cB^*F\cB}=xC+yA$; note that $\sH$ is always assumed to be a complex Hilbert space.
Since $C$ is bounded selfadjoint, it follows that $\cB{\uphar \dom \cB^*F\cB}$ is selfadjoint and therefore, in particular, closed.
By assumption the operator $\cB$,  and hence also $F^{1/2} \cB$, is unbounded. This implies that $\dom \cB^*F\cB \!\subsetneq\! \dom \cB$ by \cite[Thm.\ 3.3]{MR760619} and thus $\cB{\uphar \dom \cB^*F\cB}$ is not closed, a contradiction.
\end{proof}


\section{{Operator families with $\dom T(z)\cap\dom T(z)^*=\{0\}$, $z\in \dC$}}
\label{sec:more-ex}

In this section we return to the first example of a densely defined operator $T=UA$ with $\dom T \cap \dom T^* = \{0\}$
in a separable Hilbert space, directly derived from von Neumann's theorem, Theorem \ref{vnthm},
with unbounded selfadjoint $A$ and unitary $U$.

The following more general theorem shows that such examples are not isolated,
in the sense that there exists a strongly continuous operator family $T(z)$, $z\in \dC$,
whose values are selfadjoint operators in a Krein space with this property; here we rely on
\cite[Thms.~3.7, 3.9 and 3.10]{MR3294393}.

\begin{theorem}
\label{new1}
Let $S$ be an unbounded selfadjoint operator in a separable Hilbert space~$\sH$. Then there exists a nonconstant operator-norm continuous family
of fundamental \vspace{-1mm}symmetries
\[
   \dC {\to \bB(\sH),} \quad z\mapsto J(z),
\]
with $J(z)=-J(-1/\bar z)$, $z\ne 0$, such that the operator \vspace{-1mm} function
\[
   T(z):=J(z)S,\quad {\dom T(z)=\dom S,} \quad z\in\dC,
\]
has the \vspace{-1mm} property
\begin{align}
\label{yjdsth}
   \dom T(z)\cap\dom T(z)^*=\{0\}, \quad z\in\dC.
\end{align}
\vspace{-1mm}Moreover,
\begin{align}
\label{yjdsth1}
   \dom T(\I x)^*\cap\dom T(\I y)^*&=\{0\},  \quad x,y\in\dR, \ x\ne y,\ 1+xy\ne 0,\\
\label{yjdsth2}
   \dom T(\zeta)^*\,\cap\,\dom T(\xi)^*\hspace{1mm}&=\{0\},  \quad|\zeta|=|\xi|=1,\ \zeta\ne\pm \xi.
\end{align}
\end{theorem}

\begin{proof}
The claims follow from \cite[Thms.~3.7, 3.9 and 3.10]{MR3294393} if we apply the former two with $A:= (S^*S\!+\!I)^{-1}\!=(S^2\!+\!I)^{-1}$ and then set $J(z)\!:=\! P_{\cM(z)} \!-\! (I\!-\!P_{\cM(z)}) \!=\! 2 P_{\cM(z)} \!-\! I$, $z\!\in\! \dC$,
where $P_{\cM(z)}$ is the orthogonal projection onto the subspace $\cM(z)$ therein. Note that $A$ is a bounded, uniformly positive selfadjoint operator with $\ran A \subset \ran A^{1/2} =\dom S \ne \sH$ since $S$ is unbounded and that $\dom T(z)^*=\dom(S J(z)^*)= J(z) \dom S$ since $J(z)$ is bounded and a fundamental symmetry for $z\in\dC$.
\end{proof}

\begin{remark}
We mention that the operator family $T(z)$, $z\in \dC$, in Theorem \ref{new1} has the interesting property that
$\dom T(z)=\dom T(z')=\dom S$ is constant for all $z,z'\in \dC$, but
the domain intersections of the adjoint operators $\dom T(z)^* \cap \dom T(z')^*$ are trivial for certain pairs of $z,z'\in \dC$
by \eqref{yjdsth1}, \eqref{yjdsth2}.
\end{remark}

In general, selfadjoint operators in Krein spaces may have empty resolvent set \cite[Sect.~1.2]{MR0341174}.
Proposition \ref{prprt1} enables us to construct operator families $\wh T(z)$, $z\in \dC$,
for which $\wh T(z)$ has full numerical range, but non-empty resolvent set or even compact resolvent and for which
the properties \eqref{yjdsth}, \eqref{yjdsth1}, \eqref{yjdsth2} in Theorem \ref{new1} still hold.

\begin{corollary}
\label{newnumr}
Let $S$ and $J(z)$, $z\in \dC$, be as in Theorem {\rm \ref{new1}}. Suppose, in addition, that $S$ is chosen such that $0\in \rho(S)$,
and let the operators $\wh T(z)$, $z\in\dC$, be of the form
\[
  \wh T(z):=\begin{bmatrix}J(z)S&X\cr 0&-J(z)S  \end{bmatrix} \quad \mbox{or} \quad
	\wh T(z) =\begin{bmatrix}J(z)S&0\cr Y&-J(z)S  \end{bmatrix}
\]
with $X$, $Y\in\bB(\sH)$. Then
\[
   \dom \wh T(z)\cap \dom \wh T(z)^*=\{0\}, \quad  W\big(\wh T(z)\big)=W\big(\wh T(z)^*\big)=\dC, \quad
	 0\in\rho\big(\wh T(z)\big),
	 \qquad z\in\dC;
\]
moreover, if $S$ is chosen to have compact resolvent, then so does $\wh T(z)$, $z\in\dC$.
\end{corollary}

\begin{proof}
All claims follow from Proposition \ref{prprt1} if we note that $T(z)=J(z)S$, $z\in\dC$, cannot be of the
form $T(z)=\eta(A+\I C)$ with a symmetric operator $A$, a bounded selfadjoint operator $C$ and $\eta\in\dC$ since,
by \eqref{yjdsth}, $\dom T(z)\cap\dom T(z)^*=\{0\}$ is not dense, see Remark \ref{corso}.
\end{proof}


\section{Holomorphic families of $m$-sectorial operators and domain intersections}
\label{last}

In this last section we show that the extreme phenomenon of $m$-sectorial operators having
domain intersection $\{0\}$ with its adjoint is not isolated. There are classes of holomorphic families $T(z)$,
$z\in S(\frac {\pi}{2}- \alpha)$, of type (B), see \cite[Sect.~VII.4.2]{MR1335452}, associated with holomorphic families of sectorial forms defined in some sector $S(\frac {\pi}{2}- \alpha) \subset \dC_+$ for which the domain intersections $\dom T(z) \cap \dom T(\zeta)$, $z \ne \zeta$, and
$\dom T(z) \cap \dom T(z)^*$ may be dense or $\{0\}$.

We mention that the density of domain intersections $\bigcap_{t\in\dR} \dom (T(t)) $for families $T(t)$, $t\in\dR$, of closed densely defined
operators plays a role in measurability properties of such families, see \cite[Lemma 4.5]{MR2897729}.

\begin{theorem}\label{selfmam}
Let $L$ be an unbounded non-negative selfadjoint operator in a Hilbert space $\sH$ with $\ker L=\{0\}$ and let
$G$ be a non-negative bounded selfadjoint operator in $\sH$ with
\begin{equation}\label{hth1}
   \ker G=\{0\}, \quad \ran G\cap\dom L=\{0\}.
\end{equation}
Then the operators
\[
   \Lambda(z):=L(I+zG)L, \quad \RE z\ge 0,
\]
are $m$-sectorial and form a holomorphic family of type {\rm (B)} in the open right half-plane~$\dC_+$~with
\begin{alignat}{3}
\label{hth2}
&\dom\Lambda(\xi)\cap\dom\Lambda(\mu)&&=\{0\}, &  \quad \xi,\mu \in \dC,& \; \RE\xi\ge 0,\;\RE\mu\ge 0, \; \xi\ne\mu,\\
\label{hth3}
&\dom\Lambda(z)\cap\dom\Lambda(z)^*&&=\{0\},   &  \quad z \in \dC,& \;\RE z\ge 0,\;\IM z\ne 0.
\end{alignat}
\end{theorem}

\begin{proof}
First we show that the operator $B(z)=I+zG$ is $m$-sectorial and coercive if $\RE z\ge 0$.
For $f\in H$ and $z=\I t$, $t\in\dR$, we have
\begin{align*}
  &\RE (B(\I t)f,f)=\|f\|^2,\\
  &\left|\IM (B(\I t)f,f)\right|=|t|\,(Gf,f)\le |t| \,\|G\| \,\|f\|^2 = |t| \,\|G\| \,\RE (B(\I t)f,f),
\end{align*}
while for $z=r\exp(\I \varphi)$, $r>0$, $\varphi\in(- \frac \pi 2, \frac \pi 2)$,
\begin{align*}
  & \RE(B(z)f,f) =\|f\|^2+\RE z\, (Gf,f)\ge \|f\|^2, \\
  & \left|\IM(B(z)f,f)\right|= \frac{\IM z}{\RE z} \RE z \, (Gf,f) \le r\tan|\varphi| \,\RE (B(z)f,f).
\end{align*}
Note that $B(x)$ is a selfadjoint non-negative operator if $x\ge 0$.

Being $m$-sectorial and coercive, the operators $B(z)$, $z\in\dC_+$, give rise to closed sectorial sesquilinear forms
\[
   \sb_L(z)[h,g]:=(B(z)Lh,Lg)=\left((I+zG)Lh, Lg\right),\quad h,g\in\dom L.
\]
Because the form $z\mapsto \sb_L(z)$ is holomorphic on $\dC_+$, $\sb_L(z)$, $z\in \dC_+$, is closed on $\dom L$ and
the operator $\Lambda(z)=LB(z)L$ is associated with $\sb_L(z)$ by the first representation theorem,
the family $\Lambda(z)$ is holomorphic of type (B). Using the equality
\[
    \dom \Lambda(z)=\left\{h\in\dom L: (I+z G)Lh\in\dom L\right\}
\]
together with \eqref{hth1} and $\ker L= \ker G= \{0\}$,  we conclude that $y \in \dom\Lambda(\xi)\cap\dom\Lambda(\mu)$ satisfies
$(\xi-\mu)GLy \in \dom L$ and hence $y=0$ if $\xi\ne \mu$, which proves \eqref{hth2}.
Since $\Lambda(z)^*=\Lambda(\bar z)$, \eqref{hth2} implies \eqref{hth3}.
\end{proof}

\begin{theorem}\label{holsemgr}
Let $A$ be an $m$-$\alpha$-sectorial operator in a Hilbert space $\sH$ with $\alpha \!\in\! [0,\frac \pi 2)$ and  \vspace{-2mm} let
\[
   T(z):=\exp(-zA), \quad z\in\cS\left(\frac{\pi}2\!-\!\alpha\right):=\left\{z\in\dC:|\arg z| < \cfrac{\pi}{2}\!-\! \alpha\right\}
\]	
be the holomorphic contractive semigroup generated by $-A$, see \cite[Thm. IX.1.24]{MR1335452}. Then
\begin{equation}\label{psifi}
   \Psi(z):=A^*\left(I+T(z)\right)A, \quad \Phi(z):= A^*\left(I+T(z)\right)^{-1}A, \quad z\in \mathcal{S}\left(\frac{\pi}2\!-\!\alpha\right),
\end{equation}
are $m$-$(\alpha\!+\!|\arg z|)$-sectorial operators and form holomorphic families of type {\rm (B)} with
\begin{equation}
\label{yjdjt}
    \dom \Psi(z)^*=\dom \Phi(z)^*=\dom A^*A, \quad z\in \mathcal{S}\left(\frac{\pi}2\!-\!\alpha\right),
\end{equation}
and hence both $\Psi(\bar z)^*$ and $\Phi(\bar z)^*$, $z\!\in\! \mathcal{S}(\frac{\pi}2\!-\!\alpha)$,
form holomorphic families of type {\rm (A)}, see \cite[Sect.~VII.2]{MR1335452}. Moreover
\begin{enumerate}
\item[{\rm i)}] 
if $\,\dom A=\dom A^*$,
\vspace{-2mm}then
\begin{equation}\label{semgr3}
   \dom \Psi(z) \!=\! \dom \Psi(z)^* \!\!=\! \dom \Phi(z) \!=\! \dom \Phi(z)^* \!\!=\!\dom A^*A,
   \quad z\!\in\! \cS\left(\frac{\pi}2\!-\!\alpha\right), \hspace{-10mm} \vspace{-1mm}
\end{equation}
and both $\Psi(z)$ and $\Phi(z)$, $z\!\in\! \mathcal{S}(\frac{\pi}2\!-\!\alpha)$, form holomorphic families of type \vspace{2mm} {\rm (A)};
\item[{\rm ii)}] 
if $\,\dom A\cap\dom A^*=\{0\}$,
then
\begin{equation}
\label{semgr4}
   \dom\Psi(\xi)\cap\dom\Psi(\mu)^*\!\!=\! \dom\Phi(\xi)\cap\dom\Phi(\mu)^*\!\!=\!\{0\}, \quad \xi,\mu\!\in\!\mathcal{S}\left(\frac{\pi}2\!-\!\alpha\right).
	\hspace{-6mm}
\end{equation}
\end{enumerate}
\end{theorem}

\begin{proof}
First we note that, since $A$ is $m$-sectorial, $A^*$ is $m$-sectorial as well with corresponding holomorphic contractive semigroup $T^*(\zeta):= \exp(-\zeta A^*)$,
$\zeta\in \cS(\frac \pi 2\!-\!\alpha)$, satisfying $T(z)^*=\exp(-\overline z A^*) = T^*(\overline z)$,  $z\in \cS(\frac \pi 2\!-\!\alpha)$.
By \cite[Sect.~IX.1.6, (1.53), Rem.~1.20]{MR1335452}, we have
\begin{equation}\label{hol}
   \cfrac{{\rm d}^n T(z)}{{\rm d}z^n}=(-1)^nA^nT(z), \quad  n\in\dN, \; z\in \mathcal{S}\left(\frac{\pi}2\!-\!\alpha\right),
\end{equation}
and similary for $T^*(\zeta)= \exp(-\zeta A^*)$, $\zeta\in \cS(\frac \pi 2\!-\!\alpha)$.
Because $T(z)$ and its derivatives are bounded and hence everywhere defined, and analogously for $T^*(\zeta)$, it follows that
\begin{equation}
\label{ran-sgr}
   \hspace*{-3mm}
   \ran T(z)\!\subset\!\bigcap\limits_{n\in\dN}\dom A^n\!\subset\!\dom A, \quad
	 \ran T(z)^*\!\!\subset\!\bigcap\limits_{n\in\dN}\dom A^{*n}\!\subset\!\dom A^*\!, \quad z\!\in\! \cS\!\left(\frac{\pi}2\!-\!\alpha\right).
\end{equation}
Since $z\mapsto T(z)$ is holomorphic, \eqref{hol}  yields that
\begin{equation}
\label{kernel-sgr}
    \ker T(z)=\ker T(z)^*=\{0\}, \quad z\!\in\! \cS\big(\frac \pi 2\!-\!\alpha\big).
\end{equation}
In fact, if $z_0 \!\in\! \cS\left(\frac{\pi}2\!-\!\alpha\right)$ and $T(z_0)f=0$ for $f\in H$, then \eqref{hol} yields that $T(z)f=0$ in some neighbourhood of
$z_0$ and hence $T(z)f=0$ for all $z\!\in\! \cS\big(\frac \pi 2\!-\!\alpha\big)$ by the identity theorem. Now
$\lim_{z\to 0 \atop \!\!\!\!\!z\in\cS(\frac \pi 2 \!-\! \alpha)\!\!\!} T(z)f = f$ shows that $f=0$.

Using the spectral mapping theorem for holomorphic semigroups, see \cite[Cor.~IV.3.12 (iii)]{MR1721989}\footnote{Note that $-A$ is sectorial of angle $\frac \pi 2 \!-\!\alpha$ in the sense of \cite[Def.~II.4.1~(iii)]{MR1721989}.}, \cite[Thm.~6.4]{MR2132380}, one can prove that
\begin{equation}
\label{spec-mapp}
    \exp(-z\sigma(A))=\sigma(T(z))\setminus\{0\}, \quad z \!\in\! \cS\big(\frac \pi 2\!-\!\alpha\big).
\end{equation}
To see this, let $z_0 \!\in\! \cS\big(\frac \pi 2\!-\!\alpha\big)$ be arbitrary and write $z_0=t_0\e^{\I\psi}$ with $t_0 \in [0,\infty)$, $|\psi| < \frac \pi 2\!-\!\alpha$.
Then $\widetilde A:=\e^{\I\psi} A$ is $m$-sectorial with semi-angle $\alpha+|\psi|<\frac \pi 2$. The corresponding semigroup
$\widetilde T(z):= \exp(-z\widetilde A) = \exp( -z \e^{\I\psi}A)$, $z \!\in\!\cS(\frac \pi 2 \!-\! (\alpha \!+\!|\psi|))$,
is holomorphic in $\cS(\frac \pi 2 \!-\! (\alpha\!+\!|\psi|))$ and \cite[Cor.~IV.3.12~(iii)]{MR1721989} applied to $-\widetilde A$ and $\widetilde T$ yields that
\[
   \exp(-t_0\e^{\I\psi}\sigma(A))=
   \exp(-t_0\sigma(\widetilde A))=\sigma(\widetilde T(t_0))\setminus\{0\} =
	 \sigma(\exp(-t_0\e^{\I\psi} A))\setminus\{0\}.  	
\]
Since $A$ is $m$-sectorial, we have $-1 \!\notin\! \exp(-z\sigma(A))$ and hence
$-1\in\rho(T(z))$ for all $z\!\in\!\cS(\frac \pi2\!-\!\alpha)$.
Altogether we have shown that $\ran (I+T(z))=H$ and $\ker T(z)=\{0\}$ for $z\!\in\!\cS(\frac \pi 2\!-\!\alpha)$.

For an arbitrary contraction $K$ in a Hilbert space $H$ with real part $\RE K\!=\!(K\!+\!K^*)/2$ we claim that
\begin{equation}
\label{verylast}
   \ran (I+K)=H  \implies \ran (I+\RE K)=H;
\end{equation}
in fact, one can show that even equivalence holds.  To prove \eqref{verylast} we note that, since the
eigenvectors of $K$ and $K^*$ at $-1$ coincide, see e.g.\ \cite[\S~2]{MR0016556},
we have $\ker (I+K)= \ker (I+K^*) = \ran (I+ K)^\perp = \ran(I+K^*)^\perp$ and hence
$\ran (I+K)=H$ implies that $-1\in\rho(K)$ and, further, $-1 \in \rho(K^*)$. Then
the left hand side of the equality
\begin{equation}
\label{ineq1}
   (I+K^*)(I+K)+I-K^*K=2(I+\RE K)
\end{equation}
is uniformly positive which implies, in particular,  that $\ran (I+\RE K)=H$.

Thus, it follows that $\ran(I\!+\!\RE T(z))\!=\!H$, and even $-1 \!\in\! \rho(\RE T(z))$, for all $z\!\in\! \cS(\frac \pi 2\!-\!\alpha)$.
This implies that, for every $z\!\in\! \cS(\frac \pi 2\!-\!\alpha)$, there exists $c(z)>0$ such that
\[
     \RE\big((I+T(z))f,f\big)\ge c(z) \|f\|^2, \quad f\in H,
\]
and hence
\[
   \left| \IM \big(T(z)f,f\big)\right| \le \|f\|^2\le \cfrac{1}{c(z)}\RE\big((I+T(z))f,f\big), \quad f\in H.
\]
This inequality shows that $I+T(z)$ is a coercive $m$-sectorial operator for each $z\!\in\! \cS(\frac \pi2\!-\!\alpha)$.
In fact, a more precise estimate was established in \cite[Thm.~1, Prop.~5]{MR925792} which yields that, for $\varphi\!\in\! [0,\frac \pi 2\!-\!\alpha)$ and $r\!>\!0$,
\[
   \left| \IM \big( T\big(r\exp(\pm \I\varphi)\big)f,f\big)\right|
   \le  \cfrac{1}{2}  \tan(\alpha+\varphi)\left( \|f\|^2-\left\|T\big(r\exp(\pm \I\varphi)\big)f\right\|^2\right), \quad f\in H.
\]
Hence taking into account \eqref{ineq1}, we conclude that
\[
   \left|\IM\big(T\big(r\exp(\pm \I\varphi)\big)f,f\big)\right| \le \tan(\alpha+\varphi) \RE\big((I+T\left(r\exp(\pm \I\varphi)\big))f,f\right), \quad
   \quad f\in H.
\]
This means that the bounded operator $I\!+\!\left(r\exp(\pm \I\varphi)\right)$ is  $m$-$(\alpha\!+\!\varphi)$-sectorial if $\varphi\!\in\! [0,\frac \pi 2\!-\!\alpha)$, and
the same is true for the inverse $\left(I\!+\!T\left(r\exp(\pm \I\varphi)\right)\right)^{-1}.$

It is not difficult to check that, for $z\!\in\! \cS(\frac \pi2\!-\!\alpha)$, the sesquilinear forms
\begin{alignat*}{2}
\psi(z)[f,g]&:=\left((I+T(z))Af,Ag\right), \quad  && f,g\in\dom A,\\
\phi(z)[f,g]&:=\left((I+T(z))^{-1}Af,Ag\right), \quad && f,g\in\dom A,
\end{alignat*}
are  sectorial, closed on $\dom A$ and holomorphic in the (open) sector $\cS(\frac \pi 2\!-\!\alpha)$.
The associated $m$-sectorial operators are the operators $\Psi(z)$ and $\Phi(z)$ given by \eqref{psifi}, respectively.  
Hence the latter form holomorphic families of type (B) with
\begin{equation}\label{domain}
\begin{array}{l}
\dom\Psi(z)=\left\{f\in \dom A: (I+T(z))Af\in\dom A^*\right\},\\[0.5mm]
\dom\Phi(z)=\left\{f\in \dom A: (I+T(z))^{-1}Af\in\dom A^*\right\},
\end{array}
\end{equation}
and, since $\ran T(z)^*\subset\dom A^*$ by \eqref{ran-sgr},
\begin{equation}\label{domain1}
\begin{array}{l}
\dom\Psi(z)^*=\left\{f\in \dom A: (I+T(z)^*)\,Af\hspace{2.2mm}\in \dom A^*\right\} = \dom A^*A,\\[0.5mm]
\dom\Phi(z)^*=\left\{f\in \dom A: (I+T(z)^*)^{-1}Af\!\in\dom A^*\right\}= \dom A^*A,
\end{array}
\end{equation}
which proves \eqref{yjdjt}.

i) Suppose that $\dom A=\dom A^*$. Then $\ran T(z)\subset\dom A=\dom A^*$ by \eqref{ran-sgr}
and $\dom \Psi(z)=\dom A^*A$ follows directly from \eqref{domain}.

If $f\!\in\!\dom\Phi(z)$, then $f\!\in\!\dom A$ and $h\!=\! (I+T(z))^{-1}Af\in\dom A^*$. Since $Af\!=\!(I+T(z))h$ and $T(z)h\!\in\!\dom A\!=\!\dom A^*$, we obtain
$Af\!\in\!\dom A^*$, i.e.\ $f\!\in\!\dom A^*A$.
Con\-versely, if $f\!\in\!\dom A^*A$ and we set $h\!:=\! (I+T(z))^{-1}Af$, then $(I+T(z))h\!=\!Af\!\in\!\dom A^*$. Now
$T(z)h\in\dom A=\dom A^*$ implies that $h\in\dom A^*$ and hence $f\in\dom \Phi(z)$.
This proves $\dom \Phi(z)=\dom A^*A$ and thus, together with \eqref{domain1}, that \eqref{semgr3} holds.

ii) Suppose that $\dom A \cap\dom A^*\!=\!\{0\}$.
First let $f\!\in\!\dom\Psi(\xi)\cap\Psi(\mu)^*$. Then \eqref{domain} and \eqref{domain1} yield that
\[
   \left\{\begin{array}{l}f\in\dom A,\\
   Af+T(\xi)Af\in\dom A^*,\\
   Af+T(\mu)^*Af\in\dom A^*,
   \end{array}\right.
\implies \
   \left\{\begin{array}{l}f\in\dom A,\\[1mm]
   T(\xi)Af-T(\mu)^*Af\in\dom A^*.
   \end{array}\right.
\]
Since $T(\mu)^*Af \!\in\!\dom A^*$ and $\ran T(\xi)\!\subset\! \dom A$, this
implies $T(\xi)Af\!\in\! \dom A \cap\dom A^*\!=\!\{0\}$.
Hence, by \eqref{kernel-sgr} and \eqref{kernel-accr}, it follows that $f\in \ker A \subset \dom A \cap \dom A^* = \{0\}$.

Now let $f\in\dom\Phi(\xi)\cap\Phi(\mu)^*$. Then \eqref{domain} and \eqref{domain1} yield that
\[
\left\{\!\!\begin{array}{l}
f\in\dom A,\\
g\!=\!(I+T(\xi))^{-1}Af\in\dom A^*\!,\\
h\!=\!(I\!+\!T(\mu)^*)^{-1}Af\!\in\!\dom A^*\!,
\end{array}\right.
\!\!\Longleftrightarrow
\left\{\!\!\begin{array}{l}f\in\dom A,\\
(I+T(\xi))g=Af,\\
(I\!+\!T(\mu)^*)h\!=\!Af,\\
g,h\in\dom A^*,
\end{array}\right.
\!\!\Longrightarrow \left\{\!\!\begin{array}{l}f\in\dom A,\\
T(\xi)g\!=\!h+T(\mu)^*h\!-\!g,\\
g,h\in\dom A^*.
\end{array}\right.
\]
Since $h\!+\!T(\mu)^*h\!-\!g\in\dom A^*$, we obtain $T(\xi)g\!\in\!\dom A  \cap \dom A^* =\{0\}$.
By \eqref{kernel-sgr}, we con\-clude $g\!=\!0$ and hence $Af\!=\!0$, which again implies $f\!=\!0$ by \eqref{kernel-accr}, as above.
This completes the proof of \eqref{semgr4}.
\end{proof}

In the last theorem we consider operator functions of the form $L^*(I+T(z))L$ and \linebreak $L^*(I+T(z))^{-1}L$
with a closed densely defined operator $L$.
Note that Theorem~\ref{holsemgr} is not a special case of Theorem~\ref{yfdmt} since $A$ therein does not satisfy the
second assumption for $L$ in \eqref{below} below.

\begin{theorem}
\label{yfdmt}
Let $A$ be an $m$-$\alpha$-sectorial operator in a Hilbert space $\sH$ with $\alpha \in [0,\frac \pi 2)$, let
\[
   T(z):=\exp(-zA), \quad z\in\cS\big(\frac \pi2\!-\!\alpha\big),
\]
be the holomorphic contractive semigroup generated by $-A$, and let
$L$ be 
a closed densely defined operator such that
\begin{equation}
\label{below}
   \ker L=\{0\}, \quad \dom L^* \cap\dom A^{1/2}_R=\{0\},
\end{equation}
\vspace{-1mm}Then
\begin{equation}
\label{OmegaTheta}
   \Omega(z):=L^*(I+T(z))L, \quad \Theta(z):=L^*(I+T(z))^{-1}L, \quad z\in\cS\big(\frac \pi2\!-\!\alpha\big),
\end{equation}
form holomorphic families of type {\rm (B)} and, if $\dom A\cap\dom A^*=\{0\}$, then
\begin{equation}\label{hth3a}
   \dom\Omega(\xi)\cap\dom\Omega(\mu)^*=\dom\Theta(\xi)\cap\dom\Theta(\mu)^*=\{0\},
	 \quad  \xi,\mu\!\in\!\cS\big(\frac \pi2\!-\!\alpha\big),
\vspace{-1mm}	
\end{equation}
\vspace{-1mm} and
\begin{equation}\label{hth4}
   \bigcap\limits_{z\in\cS(\frac \pi2-\alpha)}\!\!\!\dom\Omega(z)=\{0\}, \quad
   \bigcap\limits_{z\in\cS(\frac\pi 2-\alpha)}\!\!\!\dom\Theta(z)=\{0\}.
\end{equation}
\end{theorem}

\begin{proof}
The sesquilinear forms
\begin{alignat*}{2}
\omega(z)[f,g]&:=\left((I+T(z))Lf,Lg\right), \quad && f,g\in\dom L,\\
\theta(z)[f,g]&:=\left((I+T(z))^{-1}Lf,Lg\right), \quad && f,g\in\dom L,
\end{alignat*}
are sectorial, closed and holomorphic for $z\in\cS(\frac \pi 2\!-\!\alpha)$.
Then $\Omega(z)$ and $\Theta(z)$ in \eqref{OmegaTheta} are the $m$-sectorial operators associated with the forms
$\omega(z)$ and $\theta(z)$, respectively, and hence they form holomorphic families of type (B).

The following properties of the domains of $m$-sectorial operators follow from the representations $A=A_R^{1/2}(I+\I G) A_R^{1/2}$,
$A^*=A_R^{1/2}(I-\I G) A_R^{1/2}$ where $A_R$ is the real part of $A$ and $G$ is a bounded selfadjoint operator, see
\cite[Thm.~VI.3.2]{MR1335452} and \eqref{intersection},
\begin{align}
   \hspace{-2mm}&\dom A\!\subset\!\dom A^{1/2}_R, \quad  \dom A^*\!\subset\!\dom A^{1/2}_R, \label{dom1}\\
   \hspace{-2mm}&\dom A\cap\dom A^*\!=\!\{0\} \iff \dom A\cap\dom A_R \!=\!\{0\} \iff \dom A^*\cap\dom A_R \!=\!\{0\}. \hspace{-2mm}
	\label{dom2}
\end{align}

First let $f\in\dom\Omega(\xi)\cap\dom\Omega(\mu)^*$. Then
\[
   (I+T(\xi))Lf\in\dom L^*, \quad  (I+T(\mu)^*)Lf\in\dom L^*
\]
and hence
\[
  T(\xi)Lf-T(\mu)^*Lf\in\dom L^*.
\]
Since $T(\xi)Lf\!\in\!\dom A\!\subset\!\dom A^{1/2}_R\!$, $T(\mu)^*Lf\!\in\!\dom A^*\!\!\subset\!\dom A^{1/2}_R$ by \eqref{ran-sgr}, \eqref{dom1},
we~obtain
\[
    T(\xi)Lf-T(\mu)^*Lf\in\dom A^{1/2}_R.
\]
Because $\dom L^* \cap \dom A^{1/2}_R\!=\!\{0\}$ by assumption, it follows that $ T(\xi)Lf\!-\!T(\mu)^*Lf\!=\!0$.
More\-over,  $T(\xi)Lf\in\dom A$,  $T(\mu)^*Lf\in\dom A^*$ by \eqref{ran-sgr} and thus the assumption
$\dom A\cap\dom A^*=\{0\}$ yields $T(\xi)Lf=0$, $T(\mu)^*Lf=0$. Because $\ker T(z)=\{0\}$ and $\ker L=\{0\}$ by \eqref{kernel-sgr} and assumption,
$f=0$ follows.

Now let $h\in\dom\Theta(\xi)\cap\dom\Theta(\mu)^*$. Then
\[
   \psi:=(I+T(\xi))^{-1}Lf\in\dom L^*, \quad  \phi:=(I+T(\mu)^*)^{-1}Lf\in\dom L^*.
\]
It follows that $(I+T(\xi))\psi=Lf=(I+T(\mu)^*)\phi$ and hence, again by \eqref{ran-sgr}, \eqref{dom1},
\[
   \dom A_R^{1/2} \ni T(\xi)\psi-T(\mu)^*\phi=\phi-\psi\in\dom L^*.
\]
Since $\dom L^* \cap \dom A^{1/2}_R=\{0\}$ by assumption, we conclude that
$\phi\!=\!\psi$ and $T(\xi)\psi\!=\!T(\mu)^*\psi\in\dom A\cap\dom A^* \!=\! \{0\}$, and $f\!=\!0$ follows in the same way as above,
which proves~\eqref{hth3a}.

Next let $f\in \bigcap_{z\in\cS(\frac \pi 2-\alpha)}\dom\Omega(z)$ or,  \vspace{-1.5mm} equivalently,
\[
   (I+T(z))Lf\in\dom L^*, \quad z\!\in\!\cS\left(\frac \pi 2\!-\!\alpha\right).
\vspace{-1mm}
\]
By \eqref{ran-sgr}, \eqref{dom1}, this implies \vspace{-1mm} that
\[
  \dom A_R^{1/2} \ni(T(z_1)-T(z_2))Lf\in\dom L^*, \quad z_1, z_2\!\in\!\cS\big(\frac \pi 2\!-\!\alpha\big), \, z_1 \ne z_2.
\vspace{-1mm}
\]
Since $\dom L^* \cap \dom A^{1/2}_R=\{0\}$ by assumption, we conclude that $T(z)Lf =: g$ is constant for all
$z\!\in\!\cS\big(\frac \pi 2\!-\!\alpha\big)$ or, \vspace{-1mm} equivalently,
\[
   0=\cfrac{{\rm d} T(z)Lf}{{\rm d}z}=-AT(z)Lf, \quad z\!\in\!\cS\left(\frac \pi 2\!-\!\alpha\right).
\]
By \eqref{kernel-accr}, it follows that $T(z)Lf\in \ker A \subset \dom A \cap \dom A^* = \{0\}$ and hence
$f=0$ by \eqref{kernel-sgr} and because $\ker L = \{0\}$.

Finally, let $f\in \bigcap_{z\in\cS(\frac \pi 2-\alpha)}\dom\Theta(z)$. \vspace{-1mm} Then
\[
   h(z):=(I+T(z))^{-1}Lf\in\dom L^*, \quad z\!\in\!\cS\big(\frac \pi2\!-\!\alpha\big),
\]
and hence $(I+T(z))h(z)=Lf$ is constant for all $z\!\in\!\cS\left(\frac \pi2\!-\!\alpha\right)$. \vspace{-1mm}
Then
\[
  \dom A_R^{1/2} \ni T(z_1)h(z_1)-T(z_2)h(z_2)=h(z_2)-h(z_1)\in\dom L^*,
	\quad z_1, z_2\!\in\!\cS\big(\frac \pi 2\!-\!\alpha\big), \, z_1 \ne z_2.
\]
In a similar way as above, we conclude that $h(z)=h$, $T(z)h=g$ are constant for all $z\!\in\!\cS\big(\frac \pi2\!-\!\alpha\big)$
and $f=0$ which proves \eqref{hth4}.
\end{proof}

{\small
\bibliography{Arl_Tretter}{}

\begin{thebibliography}{10}

\bibitem{MR627724}
{\sc Arlinski\u{\i}, Y.~M.}
\newblock On the defect indices of the {H}ermitian components of unbounded
  operators.
\newblock {\em Ukrain. Mat. Zh. 33}, 4 (1981), 489--490.
\newblock Engl.\ transl.\ \emph{Ukrainian Math.\ J.\ 33}, 4 (1982), 373-374.

\bibitem{MR925792}
{\sc Arlinski\u{\i}, Y.~M.}
\newblock A class of contractions in {H}ilbert space.
\newblock {\em Ukrain. Mat. Zh. 39}, 6 (1987), 691--696, 813.
\newblock Engl.\ transl.\ \emph{Ukrainian Math.\ J.\ 39}, 6 (1987), 560--564.

\bibitem{MR1772627}
{\sc Arlinski\u{\i}, Y.~M.}
\newblock On {$M$}-accretive extensions and restrictions.
\newblock {\em Methods Funct. Anal. Topology 4}, 3 (1998), 1--26.

\bibitem{MR1670389}
{\sc Arlinski\u{\i}, Y.~M.}
\newblock On functions connected with sectorial operators and their extensions.
\newblock {\em Integral Equations Operator Theory 33}, 2 (1999), 125--152.

\bibitem{MR2240273}
{\sc Arlinski\u{\i}, Y.~M.}
\newblock Extremal extensions of a {$C(\alpha)$}-suboperator and their
  representations.
\newblock In {\em Operator theory in {K}rein spaces and nonlinear eigenvalue
  problems}, vol.~162 of {\em Oper. Theory Adv. Appl.} Birkh\"{a}user, Basel,
  2006, pp.~47--69.

\bibitem{MR3612999}
{\sc Arlinski\u{\i}, Y.~M.}
\newblock On the mappings connected with parallel addition of nonnegative
  operators.
\newblock {\em Positivity 21}, 1 (2017), 299--327.

\bibitem{MR2828331}
{\sc Arlinski\u{\i}, Y.~M., Belyi, S., and Tsekanovskii, E.}
\newblock {\em Conservative realizations of {H}erglotz-{N}evanlinna functions},
  vol.~217 of {\em Oper. Theory: Adv. Appl.}
\newblock Birkh\"{a}user/Springer Basel AG, Basel, 2011.

\bibitem{MR3696196}
{\sc Arlinski\u{\i}, Y.~M., and Popov, A.}
\newblock On {$m$}-sectorial extensions of sectorial operators.
\newblock {\em Zh. Mat. Fiz. Anal. Geom. 13}, 3 (2017), 205--241.

\bibitem{MR3294393}
{\sc Arlinski\u{\i}, Y.~M., and Zagrebnov, V.~A.}
\newblock Around the van {D}aele--{S}chm\"{u}dgen theorem.
\newblock {\em Integral Equations Operator Theory 81}, 1 (2015), 53--95.

\bibitem{MR3405900}
{\sc Bagarello, F., Inoue, A., and Trapani, C.}
\newblock Weak commutation relations of unbounded operators: nonlinear
  extensions.
\newblock {\em J. Math. Phys. 53}, 12 (2012), 123510, 13.

\bibitem{MR0222718}
{\sc Berezanski\u{\i}, Y.~M.}
\newblock {\em Expansions in eigenfunctions of selfadjoint operators}.
\newblock Translations of Mathematical Monographs, Vol. 17. American
  Mathematical Society, Providence, R.I., 1968.

\bibitem{MR4083777}
{\sc B\"{o}gli, S., Marletta, M., and Tretter, C.}
\newblock The essential numerical range for unbounded linear operators.
\newblock {\em J. Funct. Anal. 279}, 1 (2020), 108509, 49.

\bibitem{MR1885442}
{\sc Cachia, V., Neidhardt, H., and Zagrebnov, V.~A.}
\newblock Comments on the {T}rotter product formula error-bound estimates for
  nonself-adjoint semigroups.
\newblock {\em Integral Equations Operator Theory 42}, 4 (2002), 425--448.

\bibitem{MR712639}
{\sc Chernoff, P.~R.}
\newblock A semibounded closed symmetric operator whose square has trivial
  domain.
\newblock {\em Proc. Amer. Math. Soc. 89}, 2 (1983), 289--290.

\bibitem{MR1974034}
{\sc Conway, J.~B., Jin, K.~H., and Kouchekian, S.}
\newblock On unbounded {B}ergman operators.
\newblock {\em J. Math. Anal. Appl. 279}, 2 (2003), 418--429.

\bibitem{MR3940390}
{\sc Corso, R.}
\newblock Maximal operators with respect to the numerical range.
\newblock {\em Complex Anal. Oper. Theory 13}, 3 (2019), 781--800.

\bibitem{MR3488056}
{\sc Cuenin, J.-C., and Tretter, C.}
\newblock Non-symmetric perturbations of self-adjoint operators.
\newblock {\em J. Math. Anal. Appl. 441}, 1 (2016), 235--258.

\bibitem{MR257789}
{\sc David, M.}
\newblock Commutators of two operators one of which is unbounded and
  semi-normal.
\newblock {\em Ann. Mat. Pura Appl. (4) 83\/} (1969), 185--194.

\bibitem{MR1721989}
{\sc Engel, K.-J., and Nagel, R.}
\newblock {\em One-parameter semigroups for linear evolution equations},
  vol.~194 of {\em Graduate Texts in Mathematics}.
\newblock Springer-Verlag, New York, 2000.

\bibitem{MR2897729}
{\sc Gesztesy, F., Gomilko, A., Sukochev, F., and Tomilov, Y.}
\newblock On a question of {A}. {E}. {N}ussbaum on measurability of families of
  closed linear operators in a {H}ilbert space.
\newblock {\em Israel J. Math. 188\/} (2012), 195--219.

\bibitem{MR2501173}
{\sc Gesztesy, F., Malamud, M., Mitrea, M., and Naboko, S.}
\newblock Generalized polar decompositions for closed operators in {H}ilbert
  spaces and some applications.
\newblock {\em Integral Equations Operator Theory 64}, 1 (2009), 83--113.

\bibitem{MR0113146}
{\sc Gohberg, I.~C., and Kre\u{\i}n, M.~G.}
\newblock The basic propositions on defect numbers, root numbers and indices of
  linear operators.
\newblock {\em Amer. Math. Soc. Transl. (2) 13\/} (1960), 185--264.

\bibitem{MR939523}
{\sc Gomilko, A.~M.}
\newblock On domains of definition of fractional powers of accretive operators.
\newblock {\em Mat. Zametki 43}, 2 (1988), 229--236, 302.
\newblock Engl.\ transl.\ \emph{Math.\ Notes 43}, 2 (1988), 129--133.

\bibitem{MR715553}
{\sc Gubr\={e}\={e}v, G.~M., and Kovalenko, A.~I.}
\newblock A new property of the determinant of a perturbation of a weak
  contraction.
\newblock {\em J. Operator Theory 10}, 1 (1983), 39--49.

\bibitem{MR2132380}
{\sc Haase, M.}
\newblock Spectral mapping theorems for holomorphic functional calculi.
\newblock {\em J. London Math. Soc. (2) 71}, 3 (2005), 723--739.

\bibitem{MR3162253}
{\sc Hansmann, M.}
\newblock Absence of eigenvalues of non-selfadjoint {S}chr\"{o}dinger operators
  on the boundary of their numerical range.
\newblock {\em Proc. Amer. Math. Soc. 142}, 4 (2014), 1321--1335.

\bibitem{MR0282244}
{\sc Hess, P., and Kato, T.}
\newblock Perturbation of closed operators and their adjoints.
\newblock {\em Comment. Math. Helv. 45\/} (1970), 524--529.

\bibitem{MR0200725}
{\sc Hildebrandt, S.}
\newblock \"{U}ber den numerischen {W}ertebereich eines {O}perators.
\newblock {\em Math. Ann. 163\/} (1966), 230--247.

\bibitem{MR3192032}
{\sc Jab{\l}o\'{n}ski, Z.~J., Jung, I.~B., and Stochel, J.}
\newblock Unbounded quasinormal operators revisited.
\newblock {\em Integral Equations Operator Theory 79}, 1 (2014), 135--149.

\bibitem{MR1155715}
{\sc Janas, J.}
\newblock On unbounded hyponormal operators. {II}.
\newblock {\em Integral Equations Operator Theory 15}, 3 (1992), 470--478.

\bibitem{MR1307601}
{\sc Janas, J.}
\newblock On unbounded hyponormal operators. {III}.
\newblock {\em Studia Math. 112}, 1 (1994), 75--82.

\bibitem{MR0138005}
{\sc Kato, T.}
\newblock Fractional powers of dissipative operators.
\newblock {\em J. Math. Soc. Japan 13\/} (1961), 246--274.

\bibitem{MR0151868}
{\sc Kato, T.}
\newblock Fractional powers of dissipative operators. {II}.
\newblock {\em J. Math. Soc. Japan 14\/} (1962), 242--248.

\bibitem{MR1335452}
{\sc Kato, T.}
\newblock {\em Perturbation theory for linear operators}.
\newblock Classics in Mathematics. Springer-Verlag, Berlin, 1995.
\newblock Reprint of the 1980 edition.

\bibitem{MR2216946}
{\sc Kosaki, H.}
\newblock On intersections of domains of unbounded positive operators.
\newblock {\em Kyushu J. Math. 60}, 1 (2006), 3--25.

\bibitem{MR1765471}
{\sc Koshmanenko, V.}
\newblock Singular operator as a parameter of self-adjoint extensions.
\newblock {\em Operator theory and related topics, Vol. II (Odessa, 1997),
  Oper. Theory Adv. Appl. 118\/} (2000), 205--223.

\bibitem{MR3496031}
{\sc Koshmanenko, V., and Dudkin, M.}
\newblock {\em The method of rigged spaces in singular perturbation theory of
  self-adjoint operators}, vol.~253 of {\em Oper. Theory: Adv. Appl.}
\newblock Birkh\"{a}user/Springer, [Cham], 2016.

\bibitem{MR2138701}
{\sc Kouchekian, S., and Thomson, J.~E.}
\newblock The density problem for self-commutators of unbounded {B}ergman
  operators.
\newblock {\em Integral Equations Operator Theory 52}, 1 (2005), 135--147.

\bibitem{MR1383821}
{\sc Kuzhel, A.}
\newblock {\em Characteristic functions and models of nonselfadjoint
  operators}, vol.~349 of {\em Mathematics and its Applications}.
\newblock Kluwer Academic Publishers Group, Dordrecht, 1996.

\bibitem{MR0341174}
{\sc Langer, H.}
\newblock Verallgemeinerte {R}esolventen eines {$J$}-nichtnegativen {O}perators
  mit endlichem {D}efekt.
\newblock {\em J. Funct. Anal. 8\/} (1971), 287--320.

\bibitem{MR0152878}
{\sc Lions, J.-L.}
\newblock Espaces d'interpolation et domaines de puissances fractionnaires
  d'op\'{e}rateurs.
\newblock {\em J. Math. Soc. Japan 14\/} (1962), 233--241.

\bibitem{MR1028066}
{\sc Martin, M., and Putinar, M.}
\newblock {\em Lectures on hyponormal operators}, vol.~39 of {\em Operator
  Theory: Advances and Applications}.
\newblock Birkh\"{a}user Verlag, Basel, 1989.

\bibitem{MR0290169}
{\sc McIntosh, A.}
\newblock On the comparability of {$A^{1/2}$} and {$A^{\ast 1/2}$}.
\newblock {\em Proc. Amer. Math. Soc. 32\/} (1972), 430--434.

\bibitem{MR0003468}
{\sc Na\u{i}mark, M.~A.}
\newblock On the square of a closed symmetric operator.
\newblock {\em C. R. (Doklady) Acad. Sci. URSS (N.S.) 26\/} (1940), 863--867.

\bibitem{MR0003469}
{\sc Na\u{i}mark, M.~A.}
\newblock \phantom{Z}{A} complement to the paper ``{O}n the square of a closed
  symmetric operator.''.
\newblock {\em C. R. (Doklady) Acad. Sci. URSS (N.S.) 28\/} (1940), 207--208.

\bibitem{MR2806469}
{\sc Nudelman, M.~A.}
\newblock A generalization of {S}tenger's lemma to maximal dissipative
  operators.
\newblock {\em Integral Equations Operator Theory 70}, 3 (2011), 301--305.

\bibitem{MR0367708}
{\sc Okazawa, N.}
\newblock Remarks on linear {$m$}-accretive operators in a {H}ilbert space.
\newblock {\em J. Math. Soc. Japan 27\/} (1975), 160--165.

\bibitem{MR760619}
{\sc \={O}ta, S.}
\newblock Closed linear operators with domain containing their range.
\newblock {\em Proc. Edinburgh Math. Soc. (2) 27}, 2 (1984), 229--233.

\bibitem{MR338818}
{\sc Putnam, C.~R.}
\newblock Almost normal operators, their spectra and invariant subspaces.
\newblock {\em Bull. Amer. Math. Soc. 79\/} (1973), 615--624.

\bibitem{MR0016556}
{\sc Riesz, F., and v.~Sz.~Nagy, B.}
\newblock \"{U}ber {K}ontraktionen des {H}ilbertschen {R}aumes.
\newblock {\em Acta Univ. Szeged. Sect. Sci. Math. 10\/} (1943), 202--205.

\bibitem{MR695940}
{\sc Schm\"{u}dgen, K.}
\newblock On domains of powers of closed symmetric operators.
\newblock {\em J. Operator Theory 9}, 1 (1983), 53--75.

\bibitem{MR0220075}
{\sc Stenger, W.}
\newblock On the projection of a selfadjoint operator.
\newblock {\em Bull. Amer. Math. Soc. 74\/} (1968), 369--372.

\bibitem{MR1018684}
{\sc Stochel, J., and Szafraniec, F.~H.}
\newblock On normal extensions of unbounded operators. {II}.
\newblock {\em Acta Sci. Math. (Szeged) 53}, 1-2 (1989), 153--177.

\bibitem{MR3509134}
{\sc ter Elst, A. F.~M., and Sauter, M.}
\newblock Nonseparability and von {N}eumann's theorem for domains of unbounded
  operators.
\newblock {\em J. Operator Theory 75}, 2 (2016), 367--386.

\bibitem{MR2463978}
{\sc Tretter, C.}
\newblock {\em Spectral theory of block operator matrices and applications}.
\newblock Imperial College Press, London, 2008.

\bibitem{MR1581206}
{\sc v.~Neumann, J.}
\newblock Zur {T}heorie der unbeschr\"{a}nkten {M}atrizen.
\newblock {\em J. Reine Angew. Math. 161\/} (1929), 208--236.
\newblock (In German).

\bibitem{MR683379}
{\sc van Daele, A.}
\newblock On pairs of closed operators.
\newblock {\em Bull. Soc. Math. Belg. S\'{e}r. B 34}, 1 (1982), 25--40.

\bibitem{MR3310942}
{\sc Voiculescu, D.-V.}
\newblock Almost normal operators mod {H}ilbert-{S}chmidt and the {$K$}-theory
  of the {B}anach algebras {$E\Lambda(\Omega)$}.
\newblock {\em J. Noncommut. Geom. 8}, 4 (2014), 1123--1145.

\bibitem{MR1887367}
{\sc Weidmann, J.}
\newblock {\em Lineare {O}peratoren in {H}ilbertr\"{a}umen. {T}eil 1}.
\newblock Mathematische Leitf\"{a}den. [Mathematical Textbooks]. B. G. Teubner,
  Stuttgart, 2000.
\newblock Grundlagen. [Foundations].

\end{thebibliography}
\bibliographystyle{acm}
}
\end{document}